\documentclass[a4paper,12pt]{amsart}
\usepackage[utf8]{inputenc}
\usepackage[margin=2.5cm]{geometry}
\usepackage{amsmath}
\usepackage{amsfonts}
\usepackage{amssymb}
\usepackage{amsthm}
\usepackage{tikz-cd} % package for category theory diagrams
\usepackage{mathtools}
\usepackage{hyperref} % clickable references and toc (NOTE: SHOULD LOAD THIS LAST)

\theoremstyle{plain}
\newtheorem*{thm*}{Theorem}
\newtheorem{thm}{Theorem}[section]
\newtheorem{prop}[thm]{Proposition}
\newtheorem{lem}[thm]{Lemma}
\newtheorem{cor}[thm]{Corollary}

\theoremstyle{definition}

\newtheorem{dfn}[thm]{Definition}
\newtheorem{ex}[thm]{Example}
\newtheorem{rem}[thm]{Remark}
\newtheorem{ntt}[thm]{Notation}

\numberwithin{equation}{thm}

\newcommand{\beq}{\begin{equation}}
\newcommand{\eeq}{\end{equation}}

\newcommand{\ZZ}{\mathbb{Z}}

\newcommand{\NN}{\mathbb{N}}
\newcommand{\CC}{\mathbb{C}}

\newcommand{\kk}{\Bbbk}
\newcommand{\Aa}{\mathbb{A}}

\DeclareMathOperator{\ad}{ad}

\newcommand{\mf}{\mathfrak}

\newcommand{\nonzero}{\setminus \{0\}}

\DeclareMathOperator{\Ann}{Ann}

\DeclareMathOperator{\id}{id}
\DeclareMathOperator{\Hom}{Hom}

\DeclareMathOperator{\Aut}{Aut}
\DeclareMathOperator{\spn}{span}
\DeclareMathOperator{\gr}{gr}
\DeclareMathOperator{\Der}{Der}

\DeclareMathOperator{\Ext}{Ext}
\DeclareMathOperator{\codim}{codim}
\DeclareMathOperator{\Inn}{Inn}
\DeclareMathOperator{\LT}{LT}
\DeclareMathOperator{\rad}{rad}

\DeclareMathOperator{\ab}{ab}

\DeclareMathOperator{\ord}{ord}

\newcommand{\W}{W_{\geq -1}}

\newcommand{\del}{\partial}
\newcommand{\diff}[1]{\frac{d}{d#1}}

\newcommand\restr[2]{{
  \left.\kern-\nulldelimiterspace % automatically resize the bar with \right
  #1 % the function
  \vphantom{\big|}
  \right|_{#2}
  }}

\title{Derivations, extensions, and rigidity of subalgebras of the Witt algebra}
\author{Lucas Buzaglo}


\begin{document}

\maketitle

\begin{abstract}
    Let $\kk$ be an algebraically closed field of characteristic 0. We study some cohomological properties of Lie subalgebras of the Witt algebra $W = \Der(\kk[t,t^{-1}])$ and the one-sided Witt algebra $\W = \Der(\kk[t])$. In the first part of the paper, we consider finite codimension subalgebras of $\W$. We compute derivations and one-dimensional extensions of such subalgebras. These correspond to $\Ext_{U(L)}^1(M,L)$, where $L$ is a subalgebra of $\W$ and $M$ is a one-dimensional representation of $L$. We find that these subalgebras exhibit a kind of rigidity: their derivations and extensions are controlled by the full one-sided Witt algebra. As an application of these computations, we prove that any isomorphism between finite codimension subalgebras of $\W$ extends to an automorphism of $\W$.

    The second part of the paper is devoted to explaining the observed rigidity. We define a notion of ``completely non-split extension" and prove that $\W$ is the universal completely non-split extension of any of its subalgebras of finite codimension. In some sense, this means that even when studying subalgebras of $\W$ as abstract Lie algebras, they remember that they are contained in $\W$. We also consider subalgebras of infinite codimension, explaining the similarities and differences between the finite and infinite codimension situations.

    Almost all of the results above are also true for subalgebras of the Witt algebra. We summarise results for $W$ at the end of the paper.
\end{abstract}

\section*{Introduction}

Throughout, we let $\kk$ be an algebraically closed field of characteristic 0. All vector spaces are $\kk$-vector spaces.

Since its inception in the late 1960s \cite{GelfandFuks, GelfandFuks2}, the study of cohomology of infinite-dimensional Lie algebras has had a rich history. The main focus has been on the Witt algebra and related Lie algebras \cite{Goncharova, FeiginFuks, DjokovicZhao}. Another class of Lie algebras that has been of particular interest in this context are \emph{Krichever--Novikov algebras} \cite{Millionshchikov, Wagemann, Wagemann2} (for the definition and general theory of Krichever--Novikov algebras, see \cite{Schlichenmaier}). Other examples can be found in \cite{FeiginLoktev, FialowskiMillionschikov} and in Fuks's book on the cohomology of infinite-dimensional Lie algebras \cite{Fuks}. The computation of these cohomology groups is an interesting and challenging problem with connections to various areas of mathematics and physics, including representation theory and deformation theory.

In this paper, we compute the first cohomology of some infinite-dimensional Lie algebras with coefficients in the adjoint representation, which is closely related to computing derivations of these Lie algebras: for a Lie algebra $L$, we have
$$H^1(L;L) \cong \Der(L)/\Inn(L),$$
where $\Der(L)$ is the space of derivations of $L$ and $\Inn(L)$ is the space of inner derivations of $L$. The study of derivations of Lie algebras dates back to the pioneering works of Jacobson, Hochschild, Leger, Dixmier, and others \cite{Jacobson, Hochschild, Leger, DixmierLister}, and continues to be an active area of research to this day (see, for example, \cite{GaoJiangPei, FeldvossSicilianoWeigel, MoritaSakasaiSuzuki, SaeediSheikh-Mohseni}).

Let $\W = \Der(\kk[t])$ be the \emph{one-sided Witt algebra} and let $W = \Der(\kk[t,t^{-1}])$ be the \emph{Witt algebra} (or \emph{centerless Virasoro algebra}). These are infinite-dimensional Lie algebras which are naturally modules over $\kk[t]$ and $\kk[t,t^{-1}]$, respectively. We are interested in infinite-dimensional subalgebras of $\W$ and $W$. Of particular interest are those subalgebras of $\W$ and $W$ which are also $\kk[t]$- or $\kk[t,t^{-1}]$-submodules, which we call \emph{submodule-subalgebras}. They are denoted by $\W(f) \coloneqq f\W$ and $W(g) \coloneqq gW$ for $f \in \kk[t]$ and $g \in \kk[t,t^{-1}]$. All subalgebras of $\W$ of finite codimension are ``essentially" submodule-subalgebras: if $L$ is such a subalgebra, then there exists $f \in \kk[t]$ such that $\W(f) \subseteq L \subseteq \W(\rad(f))$, where $\rad(f) = \prod_{\xi \in V(f)} (t - \xi)$ \cite[Proposition 3.2.7]{PetukhovSierra}. Furthermore, by \cite{BellBuzaglo}, any infinite-dimensional subalgebra of $\W$ is isomorphic to a subalgebra of finite codimension. We can therefore extract a lot of information about infinite-dimensional subalgebras of $\W$ by first studying submodule-subalgebras.

Of the submodule-subalgebras of $\W$, those of the form $W_{\geq n} \coloneqq \W(t^{n+1})$ for $n \in \NN$ are of particular significance and have been extensively studied. For example, $W_{\geq 1}$ is sometimes called the \emph{positive Witt algebra}, which appeared in \cite{Fialowski} in the context of deformation theory, where the cohomology group $H^2(W_{\geq 1};W_{\geq 1})$ was computed. Similar results for the Lie algebras $W_{\geq 2}$ and $W_{\geq 3}$ were established in \cite{FialowskiPost} and \cite{Kochetkov}.

Submodule-subalgebras of $\W$ and $W$ form a large and interesting class of infinite-dimensional Lie algebras of linear growth. More generally, one could study submodule-subalgebras of $\Der(A)$, where $A$ is an associative $\kk$-algebra. These first appeared in \cite{Grabowski}, and subsequently in \cite{Grabowski2} and \cite{Siebert}. In particular, Siebert showed that if $L_1$ and $L_2$ are submodule-subalgebras of $\Der(\kk[X_1])$ and $\Der(\kk[X_2])$, respectively, where $X_1$ and $X_2$ are irreducible affine varieties, then $X_1$ and $X_2$ are isomorphic if $L_1$ and $L_2$ are isomorphic, provided $L_1$ and $L_2$ satisfy certain additional properties \cite[Theorem 3]{Siebert}.

We start off by computing derivations of infinite-dimensional subalgebras of $\W$.

\begin{thm}[Theorems \ref{thm:derivations of W(f)} and \ref{thm:derivations of general subalgebras}]\label{thm:intro derivations}
    Let $L$ be an infinite-dimensional subalgebra of $\W$. Then
    $$\Der(L) = \{\ad_w \mid w \in \W \text{ such that } \ad_w(L) \subseteq L\}.$$
    In other words, derivations of $L$ are restrictions of derivations of $\W$. Consequently, $H^1(L;L) \cong N_{\W}(L)/L$, where $N_{\W}(L)$ is the normaliser of $L$ in $\W$.

    In particular,
    $$\Der(\W(f)) = \{\ad_w \mid w \in \W(\rad(f))\} \cong \W(\rad(f))$$
    for $f \in \kk[t]$, where $\rad(f) = \prod_{\xi \in V(f)}(t - \xi)$. As a result,
    $$\dim H^1(\W(f);\W(f)) = \deg(f) - \deg(\rad(f)).$$
\end{thm}

For a Lie algebra $L$, we have $H^1(L;L) = \Ext_{U(L)}^1(\kk,L)$, where $\kk$ is the one-dimensional trivial representation of $L$. We can generalise Theorem \ref{thm:intro derivations} by computing
$$\Ext_{U(\W(f))}^1(M,\W(f)),$$
where $f \in \kk[t]$ and $M$ is an arbitrary one-dimensional representation of $\W(f)$. An element of $\Ext_{U(L)}^1(M,L)$ can be viewed as a non-split short exact sequence of $L$-representations
$$0 \to L \to \overline{L} \to M \to 0.$$
We study such short exact sequences in Section \ref{sec:extensions} for $L$ a subalgebra of $\W$ of finite codimension. In particular, we prove that $\overline{L}$ can be canonically embedded in $\W$ (see Proposition \ref{prop:extensions are in W}).

The above results suggest that subalgebras of $\W$ of finite codimension are ``rigid", in the sense that they inherit their properties from $\W$. We explain this by showing that $\W$ can be intrinsically built from any of its subalgebras of finite codimension as the universal \emph{completely non-split extension}, which is an abstract Lie-algebraic property (see Definition \ref{dfn:completely non-split extension}).

\begin{thm}[Theorem \ref{thm:universal property}]\label{thm:intro universal}
    Let $L$ be a subalgebra of $\W$ of finite codimension. Then $\W$ is the universal completely non-split extension of $L$, in the following sense: $\W$ is a completely non-split extension of $L$, and if $\overline{L}$ is another completely non-split extension of $L$, then $\overline{L}$ can be uniquely embedded in $\W$ such that the diagram
    \begin{center}
        \begin{tikzcd}
            L \arrow[d, hook] \arrow[r, hook] & \W \\
            \overline{L} \arrow[ru, hook]     &   
        \end{tikzcd}
    \end{center}
    commutes.
\end{thm}

Given Siebert's result, the rigidity noted above, and the close relationship between subalgebras of $\W$ of finite codimension and submodule-subalgebras, it is natural to ask whether any isomorphism between subalgebras of $\W$ of finite codimension extends to an automorphism of $\W$. If we knew of this universal property in advance, this would be immediate. In fact, we do prove this on the way to proving the universal property.

\begin{thm}[Theorem \ref{thm:extending isomorphisms}]\label{thm:intro isomorphism}
    Let $L_1$ and $L_2$ be subalgebras of $\W$ of finite codimension and suppose there is an isomorphism of Lie algebras $\varphi \colon L_1 \to L_2$. Then $\varphi$ extends to an automorphism of $\W$.
\end{thm}

The idea of the proof of Theorem \ref{thm:intro isomorphism} is as follows: let $\varphi \colon L_1 \to L_2$ be an isomorphism between proper subalgebras of $\W$ of finite codimension. We show that there exist subalgebras $\overline{L}_i$ of $\W$ with $L_i \subseteq \overline{L}_i \subseteq \W$ and $\dim(\overline{L}_i/L_i) = 1$ for $i = 1,2$ such that $\varphi$ extends to an isomorphism $\overline{\varphi} \colon \overline{L}_1 \to \overline{L}_2$. The proof then follows by induction on $\codim_{\W}(L_1)$.

In order to prove the existence of the isomorphism $\overline{\varphi} \colon \overline{L}_1 \to \overline{L}_2$, we apply our study of non-split short exact sequences of the form
$$0 \to L \to \overline{L} \to \overline{L}/L \to 0$$
where $L$ is a subalgebra of $\W$ of finite codimension and $\dim(\overline{L}/L) = 1$. In other words, Theorem \ref{thm:intro isomorphism} follows from the computation of $\Ext_{U(L)}^1(M,L)$, where $M$ is a one-dimensional representation of $L$.

We now explain the structure and proofs of the paper in more detail. There are two main goals: the first is to compute $\Ext_{U(\W(f))}^1(M,\W(f))$, where $f$ and $M$ are as above, and the second is to provide an explanation for the observed rigidity.

Sections \ref{sec:graded derivations} and \ref{sec:general submodule-subalgebras} are devoted to computing derivations of $\W(f)$ for $f \in \kk[t]$, which corresponds to the case where $M = \kk$ is the trivial representation. This is because
$$\Ext_{U(\W(f))}^1(\kk,\W(f)) = H^1(\W(f);\W(f)) \cong \Der(\W(f))/\Inn(\W(f))$$
is the first cohomology of $\W(f)$ with coefficients in the adjoint representation.

It is easy to show that all derivations of $\W$ are inner. Given the rigidity of the situation, it is natural to ask if derivations of submodule-subalgebras of $\W$ also arise from the adjoint action of elements of $\W$, that is, if
$$\Der(\W(f)) = \{\ad_w \mid w \in \W \text{ and } \ad_w(\W(f)) \subseteq \W(f)\},$$
where $f \in \kk[t]$. It is easy to see that $\ad_w(\W(f)) \subseteq \W(f)$ if and only if $w \in \W(\rad(f))$, where $\rad(f) = \prod \{t - \xi \mid f(\xi) = 0\}$. Therefore, the question is equivalent to asking whether $\Der(\W(f)) = \{\ad_w \mid w \in \W(\rad(f))\}$, which is precisely what Theorem \ref{thm:intro derivations} says.

As the Lie algebra $\W$ is graded, the proof of Theorem \ref{thm:intro derivations} starts by considering derivations of graded submodule-subalgebras, that is, Lie algebras of the form $W_{\geq n}$ for $n \in \NN$. For this situation, we can use results from \cite{Farnsteiner}. In particular, the space of derivations of a graded Lie algebra is itself graded, which greatly simplifies our computations. However, computing derivations of ungraded submodule-subalgebras requires more work. Since submodule-subalgebras of $\W$ are filtered, this is achieved by an associated graded argument, allowing us to compute derivations of ungraded submodule-subalgebras from derivations of graded ones.

In Section \ref{sec:extensions}, we move on to computing $\Ext_{U(\W(f))}^1(M,\W(f))$, where $f \in \kk[t]$ and $M$ is an arbitrary one-dimensional representation of $\W(f)$. A key observation is that the derived subalgebra of $\W(f)$ is $\W(f^2)$, so the action of $\W(f^2)$ on any one-dimensional representation of $\W(f)$ is trivial. Consequently, there are many one-dimensional representations of $\W(f)$, corresponding to the representations of the abelian Lie algebra $\W(f)/\W(f^2)$. In other words, the one-dimensional representations of $\W(f)$ are parametrised by $\Aa^{\deg(f)}$. However, the only one-dimensional modules $M$ which give rise to non-split extensions
$$0 \to \W(f) \to L \to M \to 0$$
are of the form $M = \kk$ or $M = \W/\W(t - \xi)$ for $\xi \in V(f)$ a root of order 1. The $U(\W(f))$-module structure on $\W/\W(t - \xi)$ is induced from $\W(f)$ acting by its adjoint action on both $\W$ and $\W(t - \xi)$, noting that $\W(f)$ is a Lie subalgebra of both $\W$ and $\W(t - \xi)$.

\begin{thm}[Theorem \ref{thm:one-dimensional extension}]\label{thm:intro extensions}
    Let $f \in \kk[t]$. If $M \cong \W/\W(t - \xi)$ for some $\xi \in \kk$ a root of $f$ of multiplicity 1, then
    $$\dim(\Ext_{U(\W(f))}^1(\W/\W(t - \xi),\W(f))) = 1.$$
    In this case, the unique non-split extension of $\W/\W(t - \xi)$ by $\W(f)$ is $\W(\frac{f}{t - \xi})$. If $M \cong \kk$ is trivial, we have
    $$\dim(\Ext_{U(\W(f)}^1(\kk,\W(f))) = \deg(f) - \deg(\rad(f)).$$
    Otherwise, $\Ext_{U(\W(f))}^1(M,\W(f)) = 0$.
\end{thm}

Although it is not immediately obvious, the proof is closely related to the computation of derivations of submodule-subalgebras: we consider a short exact sequence of $U(\W(f))$-modules
$$0 \to \W(f) \to L \to M \to 0$$
as a sequence of $U(\W(f^2))$-modules. Since $M$ is trivial as a representation of $\W(f^2)$, the short exact sequence gives an element of
$$\Ext_{U(\W(f^2))}^1(\kk,\W(f)) \cong \Der(\W(f^2),\W(f))/\Inn(\W(f^2),\W(f)),$$
which can be computed using the techniques from Theorem \ref{thm:intro derivations}.

In Section \ref{sec:universal}, we prove Theorem \ref{thm:intro universal}, providing an explanation for the aforementioned rigidity. We first show that $\W$ is a completely non-split extension of any of its subalgebras of finite codimension, which we achieve by proving that any such subalgebra of $\W$ has a one-dimensional extension contained in $\W$. We then prove the universal property using results from Section \ref{sec:extensions} about one-dimensional extensions of subalgebras of $\W$.

We then briefly consider subalgebras of $\W$ of infinite codimension. Any infinite-dimensional subalgebra of $\W$ is isomorphic to a subalgebra of $\W$ of finite codimension \cite{BellBuzaglo}, so our results in finite codimension can be translated to infinite codimension. However, there are also some key differences when it comes to the universal property, which we highlight at the end of Section \ref{sec:universal}.

We conclude the paper by establishing analogous results to Theorems \ref{thm:intro derivations}--\ref{thm:intro extensions} for submodule-subalgebras of the Witt algebra. The proofs for submodule-subalgebras of $\W$ go through for $W$ with the obvious changes. In fact, the situation for the Witt algebra is slightly easier: the associated graded algebra to $W(f)$ is $W$, where $f \in \kk[t,t^{-1}]$. Consequently, in order to use the associated graded argument to compute derivations of $W(f)$, we only need to consider derivations of $W$. Thanks to results from \cite{Farnsteiner}, it is easy to see that $\Der(W) = \Inn(W)$.

\vspace{2mm}

\noindent \textbf{Acknowledgements:} This work was done as part of the author's PhD research at the University of Edinburgh.

Proposition \ref{prop:codimension abelianisation} is part of a collaboration with Jason Bell carried out during the author's visit to the University of Waterloo. We are grateful to Prof. Bell for allowing us to include it.

\section{Derivations of graded submodule-subalgebras}\label{sec:graded derivations}

We begin by recalling the notions of derivations of algebras, and defining the Lie algebras of interest in this paper.

\begin{dfn}
    A \emph{derivation} of a $\kk$-algebra $A$ is a $\kk$-linear map $d \colon A \to A$ such that
    $$d(ab) = ad(b) + d(a)b$$
    for $a,b \in A$. We write $\Der(A)$ for the set of derivations of $A$.
    
    We write $\W = \Der(\kk[t]) = \kk[t]\del$ and $W = \Der(\kk[t,t^{-1}]) = \kk[t,t^{-1}]\del$ for the \emph{one-sided Witt algebra} and the \emph{Witt algebra}, respectively, where $\del = \diff{t}$. The Witt algebra is sometimes called the \emph{centerless Virasoro algebra}.
    
    We let $e_n = t^{n+1}\del \in W$ for $n \in \ZZ$. The Lie bracket is given by
    $$[f\del,g\del] = (fg' - f'g)\del, \quad [e_n,e_m] = (m - n)e_{n+m},$$
    where $f,g \in \kk[t,t^{-1}]$ and $n,m \in \ZZ$. Note that $\W$ is spanned by $\{e_n \mid n \geq -1\}$.
    
    For $f \in \kk[t]$, we write $\W(f) = f\W = f\kk[t]\del$ under the obvious notation. It is clear that $\W(f)$ is a subalgebra of $\W$, as well as a $\kk[t]$-submodule. We call $\W(f)$ a \emph{submodule-subalgebra} of $\W$. We write $W_{\geq n}$ for $\W(t^{n+1})$.

    We will also consider the Lie algebra $\Der(\kk(t)) = \kk(t)\del$, and view $\W$ and $W$ as subalgebras of $\kk(t)\del$.
\end{dfn}

Our main goal in the first part of the paper is to compute $\Ext^1_{U(\W(f))}(M,\W(f))$, where $f \in \kk[t] \nonzero$ and $M$ is a one-dimensional $U(\W(f))$-module. We devote the first two sections to studying the easiest case, where $M = \kk$ is the trivial one-dimensional representation of $\W(f)$. In this case, $\Ext^1_{U(\W(f))}(\kk,\W(f))$ is a well-known object.

\begin{dfn}
    For a Lie algebra $L$ and a representation $M$ of $L$, the \emph{cohomology of $L$ with values in $M$}, denoted $H^n(L;M)$ for $n \in \NN$, is defined as
    $$H^n(L;M) = \Ext^n_{U(L)}(\kk,M).$$
\end{dfn}

We are therefore interested in computing $H^1(\W(f);\W(f))$ for $f \in \kk[t] \nonzero$. It is a standard fact that, for a Lie algebra $L$ and a representation $M$ of $L$, the cohomology group $H^1(L;M)$ is isomorphic to the quotient $\Der(L,M)/\Inn(L,M)$, defined below.

\begin{dfn}
    A \emph{derivation} of a Lie algebra $L$ with values in an $L$-module $M$ is a linear map $d \colon L \to M$ such that
    $$d([x,y]) = x \cdot d(y) - y \cdot d(x)$$
    for $x,y \in L$. We write $\Der(L,M)$ for the set of derivations $L$ with values in $M$, and simply write $\Der(L)$ instead of $\Der(L,L)$. A derivation $d \in \Der(L,M)$ is \emph{inner} if there exists $m \in M$ such that $d(x) = x \cdot m$. We write $\Inn(L,M)$ for the set of inner derivations of $L$ with values in $M$. As before, we write $\Inn(L)$ instead of $\Inn(L,L)$.
    
    If $L = \bigoplus_{n \in \ZZ}L_n$ is a $\ZZ$-graded Lie algebra and $M = \bigoplus_{n \in \ZZ} M_n$ is a $\ZZ$-graded representation of $L$, then for $k \in \ZZ$ we write
    $$\Der(L,M)_k = \{d \in \Der(L,M) \mid d(L_n) \subseteq M_{n+k} \text{ for all } n \in \ZZ\}.$$
    If $d \in \Der(L,M)_k$, we say that $d$ is a \emph{graded derivation of degree $k$}.
\end{dfn}

If $L$ is a Lie algebra and $M$ is a representation of $L$, we can construct elements of $\Ext_{U(L)}^1(\kk,M)$ from derivations of $L$ with values in $M$ as follows: given a derivation $d \in \Der(L,M)$, consider the vector space $X = M \oplus \kk d$ with $L$-action given by the usual action of $L$ on $M$ and $w \cdot d = -d(w)$ for $w \in L$. Note that $X/M$ is trivial as a representation of $L$, so we get a short exact sequence of representations of $L$ given by
\beq\label{eq:derivation gives extension}
    0 \to M \to X \to \kk \to 0.
\eeq
Furthermore, we easily see that \eqref{eq:derivation gives extension} splits if and only if $d \in \Inn(L,M)$.

We therefore see that computing $H^1(\W(f);\W(f)) = \Ext_{U(\W(f))}^1(\kk,\W(f))$ is equivalent to computing derivations of $\W(f)$. Thanks to the following result of Farnsteiner, studying derivations of graded Lie algebras is easier than ungraded ones, since we can exploit the graded structure of the space of derivations.

\begin{prop}[{\cite[Proposition 1.1]{Farnsteiner}}]\label{prop:decomposition of Der}
    If $L = \bigoplus_{n \in \ZZ}L_n$ is a $\ZZ$-graded finitely generated Lie algebra and $M = \bigoplus_{n \in \ZZ} M_n$ is a $\ZZ$-graded representation of $L$, then
    $$\Der(L,M) = \bigoplus_{n \in \ZZ}\Der(L,M)_n.$$
\end{prop}

The Lie algebras $W$ and $\W$ are $\ZZ$-graded, with $e_n$ having degree $n$. However, submodule-subalgebras of $\W$ are not graded in general: the graded ones are $W_{\geq n}$ for $n \geq -1$. Still, ungraded submodule-subalgebras are filtered, with associated graded algebras isomorphic to subalgebras of the form $W_{\geq n}$. It is therefore natural to begin our study of derivations of submodule-subalgebras of $\W$ by considering the graded ones, and later attempt to extract as much information as we can for general submodule-subalgebras via an associated graded argument.

The goal for this section is to compute $\Der(W_{\geq n},W)$ for $n \in \NN$, where we view $W$ as a representation of $\W$ under the adjoint action. We start with the cases $n = -1,0,1$. The computation of derivations of $\W$ was already a well-known folklore result, but we still give a proof here. However, derivations of $W_{\geq n}$ for arbitrary $n$ were not known before.

\begin{prop}\label{prop:derivations of W1}
    We have $\Der(\W,W) = \Inn(\W,W)$. Consequently, $\Der(\W) = \Inn(\W) \cong \W$ and $H^1(\W;\W) = 0$.
\end{prop}

We will use a theorem of Farnsteiner to prove Proposition \ref{prop:derivations of W1}.

\begin{thm}[{\cite[Proposition 1.2]{Farnsteiner}}]\label{thm:Farnsteiner}
   Let $L = \bigoplus_{n \in \ZZ} L_n$ be a $\ZZ$-graded finitely generated Lie algebra and let $M = \bigoplus_{n \in \ZZ} M_n$ be a graded representation of $L$ such that
   \begin{enumerate}
       \item $H^1(L_0;M_n) = 0$ for all $n \in \ZZ \nonzero$.
       \item $\Hom_{L_0}(L_n,M_m) = 0$ for $n \neq m$.
   \end{enumerate}
   Then $\Der(L,M) = \Der(L,M)_0 + \Inn(L,M)$.
\end{thm}

In order to compute $\Der(\W,W)$, we show that $\Der(\W,W)_0 \subseteq \Inn(\W,W)$. It will then be clear that $\Der(\W,W) = \Inn(\W,W)$, provided the conditions of Theorem \ref{thm:Farnsteiner} are satisfied.

\begin{lem}\label{lem:derivations of degree 0}
    Let $n \geq -1$. We have $\Der(W_{\geq n},W)_0 = \kk \ad_{e_0} \subseteq \Inn(\W,W)$.
\end{lem}
\begin{proof}
    Let $d \in \Der(\W,W)_0$. Then for all $n \geq -1$ there exists $\lambda_n \in \kk$ such that $d(e_n) = \lambda_n e_n$. We have
    $$d([e_n,e_m]) = [e_n,d(e_m)] + [d(e_n),e_m] = (\lambda_n + \lambda_m)[e_n,e_m] = (m - n)(\lambda_n + \lambda_m)e_{n+m}.$$
    On the other hand,
    $$d([e_n,e_m]) = (m - n)d(e_{n+m}) = (m - n)\lambda_{n+m}e_{n+m}.$$
    Combining the above, we conclude that $\lambda_n + \lambda_m = \lambda_{n+m}$ for all $n \neq m$. Now, $\lambda_0 + \lambda_1 = \lambda_1$, so $\lambda_0 = 0$. Furthermore, $\lambda_{-1} + \lambda_1 = \lambda_0 = 0$, so $\lambda_{-1} = -\lambda_1$.
    
    Now consider $2\lambda_1 + \lambda_3$:
    $$2\lambda_1 + \lambda_3 = \lambda_1 + \lambda_4 = \lambda_5 = \lambda_2 + \lambda_3.$$
    Hence, $\lambda_2 = 2\lambda_1$. Inductively, we deduce that $n\lambda_1 = \lambda_n$ for all $n \geq -1$. Letting $\lambda = \lambda_1$, it follows that $d = \lambda \ad_{e_0} \in \Inn(\W,W)$.
\end{proof}

\begin{proof}[Proof of Proposition \ref{prop:derivations of W1}]
    By Lemma \ref{lem:derivations of degree 0}, it suffices to check the conditions for Theorem \ref{thm:Farnsteiner} are satisfied when $L = \W$ and $M = W$.
    
    First, we must show that 
    $$H^1(\kk e_0;\kk e_n) \cong \Der(\kk e_0,\kk e_n)/\Inn(\kk e_0,\kk e_n) = 0$$
    for all $n \in \ZZ \nonzero$. Let $n \in \ZZ \nonzero$. We have
    $$\Inn(\kk e_0,\kk e_n) = \kk \ad_{e_n} \subseteq \Der(\kk e_0,\kk e_n) \subseteq \Hom_\kk(\kk e_0,\kk e_n).$$
    By a dimension count, $H^1(\kk e_0;\kk e_n) = 0$.

    To end the proof, we can easily see that $\Hom_{\kk e_0}(\kk e_n,\kk e_m) = 0$ for $n \neq m$, where $n \in \ZZ_{\geq -1}$ and $m \in \ZZ$, since $e_n$ and $e_m$ have different eigenvalues with respect to the action of $e_0$.
\end{proof}

The computation of derivations of $W_{\geq 0}$ is almost identical to derivations of $\W$. We therefore omit the proof.

\begin{prop}\label{prop:derivations of W0}
    We have $\Der(W_{\geq 0},W) = \Inn(W_{\geq 0},W)$. Consequently, $\Der(W_{\geq 0}) = \Inn(W_{\geq 0}) \cong W_{\geq 0}$ and $H^1(W_{\geq 0};W_{\geq 0}) = 0$. \qed
\end{prop}

Proposition \ref{prop:derivations of W0} shows that $\Ext_{U(W_{\geq 0})}^1(\kk,W_{\geq 0}) = 0$. However, $W_{\geq 0}$ does have nontrivial extensions by one-dimensional representations: consider the short exact sequence
\beq\label{eq:extension of W0}
    0 \to W_{\geq 0} \to \W \to \W/W_{\geq 0} \to 0.
\eeq
This is because the $U(W_{\geq 0})$-module $\W/W_{\geq 0}$ is not the trivial one-dimensional module, since $e_0$ acts as $-\id$. In fact, we will see later that $\dim(\Ext_{U(W_{\geq 0})}^1(\W/W_{\geq 0},W_{\geq 0})) = 1$ and that \eqref{eq:extension of W0} is the unique nontrivial one-dimensional extension of $W_{\geq 0}$ (see Theorem \ref{thm:one-dimensional extension}).

Computing $\Der(W_{\geq 1},W)$ requires more work, since $e_0 \not\in W_{\geq 1}$, so we cannot use Theorem \ref{thm:Farnsteiner}. Instead, we can use the graded structure of $\Der(W_{\geq 1})$ from Proposition \ref{prop:decomposition of Der}.

\begin{prop}\label{prop:W+}
    We have $\Der(W_{\geq 1},W) = \Inn(W_{\geq 1},W)$. Consequently, $\Der(W_{\geq 1}) = \{\ad_w \mid w \in W_{\geq 0}\} \cong W_{\geq 0}$ and $\dim(H^1(W_{\geq 1};W_{\geq 1})) = 1$.
\end{prop}
\begin{proof}
    By Proposition \ref{prop:decomposition of Der}, it suffices to prove that $\Der(W_{\geq 1},W)_n \subseteq \Inn(W_{\geq 1},W)$ for $n \in \ZZ$.
    
    The same proof as in Lemma \ref{lem:derivations of degree 0} works to show that $\Der(W_{\geq 1},W)_0 = \kk \ad_{e_0}$.
    
    We start by considering $\Der(W_{\geq 1},W)_1$. Note that, since $W_{\geq 1}$ is generated by $e_1$ and $e_2$ as a Lie algebra, any derivation $d \in \Der(W_{\geq 1},W)$ is uniquely determined by $d(e_1)$ and $d(e_2)$. Let $d \in \Der(W_{\geq 1},W)_1$, so $d(e_1) = \lambda e_2$ and $d(e_2) = \mu e_3$ for some $\lambda, \mu \in \kk$. Let $d' = d - \mu\ad_{e_1} \in \Der(W_{\geq 1},W)$, so that $d'(e_1) = \lambda e_2$ and $d'(e_2) = 0$. We have
    \begin{align*}
        d'(e_3) &= d'([e_1,e_2]) = [e_1,d'(e_2)] + [d'(e_1),e_2] = \lambda[e_2,e_2] = 0, \\
        d'(e_4) &= \frac{1}{2}d'([e_1,e_3]) = \frac{1}{2}([e_1,d'(e_3)] + [d'(e_1),e_3]) = \frac{1}{2}\lambda[e_2,e_3] = \frac{1}{2}\lambda e_5, \\
        d'(e_5) &= \frac{1}{3}d'([e_1,e_4]) = \frac{1}{3}([e_1,d'(e_4)] + [d'(e_1),e_4]) = \frac{1}{3}(\frac{1}{2}\lambda[e_1,e_5] + \lambda[e_2,e_4]) = \frac{4}{3}\lambda e_6.
    \end{align*}
    On the other hand,
    $$d'(e_5) = d'([e_2,e_3]) = [e_2,d'(e_3)] + [d'(e_2),e_3] = 0.$$
    Therefore, $\lambda = 0$, so $d' = 0$. It follows that $d = \mu \ad_{e_1}$, so $\Der(W_{\geq 1},W)_1 \subseteq \Inn(W_{\geq 1},W)$.

    Now let $d \in \Der(W_{\geq 1},W)_2$, so $d(e_1) = \lambda e_3$ and $d(e_2) = \mu e_4$ for some $\lambda, \mu \in \kk$. Consider the subalgebra $V_2$ of $\W$ spanned by $e_2,e_4,e_6,\ldots$. Certainly, $d$ restricts to a derivation of $V_2$. Note that $V_2$ is isomorphic to $W_{\geq 1}$ via the map $e_{2n} \mapsto 2e_n$. Under this isomorphism, $d$ induces a derivation $\widetilde{d}$ of $W_{\geq 1}$ of degree 1. By the above computation of $\Der(W_{\geq 1},W)_1$, we know that $\widetilde{d}$ is a scalar multiple of $\ad_{e_1}$, and therefore $\restr{d}{V_2}$ is a scalar multiple of $\ad_{e_2}$. In particular, we see that $d(e_2) = 0$, so that $\mu = 0$. It follows that $d = -\lambda \ad_{e_2}$, so $\Der(W_{\geq 1},W)_2 \subseteq \Inn(W_{\geq 1},W)$.
    
    Finally, let $n \in \ZZ \setminus \{0,1,2\}$ and consider $d \in \Der(W_{\geq 1},W)_n$. Then $d(e_1) = \lambda e_{n+1}$ and $d(e_2) = \mu e_{n+2}$ for some $\lambda, \mu \in \kk$. Taking Lie brackets up to $e_5$ and using the Leibniz rule for $d$ as above, we get
    $$(\lambda - \mu)n^3 - \lambda n^2 + (4\lambda - 5\mu)n - 12\lambda + 6\mu = 0.$$
    Let $\alpha = \frac{1}{n - 1}\lambda$ and $\beta = \frac{1}{n - 2}\mu$. Substituting $\alpha$ and $\beta$ into the equation above,
    $$(n^4 - 2n^3 + 5n^2 - 16n + 12)(\alpha - \beta) = 0.$$
    The quartic equation in $n$ does not vanish for $n \geq 3$ (its only integer solutions are $n = 1,2$), so we conclude that $\alpha = \beta$. Therefore, $d(e_1) = \alpha(n - 1)e_{n+1}$ and $d(e_2) = \alpha(n - 2)e_{n+2}$. It follows that $d = -\alpha \ad_{e_n}$, so $\Der(W_{\geq 1},W)_n \subseteq \Inn(W_{\geq 1},W)$ for all $n \geq 3$. This finishes the proof.
\end{proof}

We now move on to computing $\Der(W_{\geq n},W)$ for $n \geq 2$.

\begin{prop}\label{prop:derivations of Wn}
    For $n \in \NN$, we have $\Der(W_{\geq n},W) = \Inn(W_{\geq n},W)$. Consequently, $\Der(W_{\geq n}) = \{\ad_w \mid w \in W_{\geq 0}\} \cong W_{\geq 0}$ and $\dim(H^1(W_{\geq n};W_{\geq n})) = n$.
\end{prop}

In order to prove Proposition \ref{prop:derivations of Wn}, it suffices to compute graded derivations of $W_{\geq n}$, by Proposition \ref{prop:decomposition of Der}. The idea of the proof is similar to that of Proposition \ref{prop:derivations of W1}. In that proof, we used the fact that we can take brackets up to $e_5$ in two different ways, which puts restrictions on the space of derivations of $W_{\geq 1}$. In other words, we exploited the relation
$$[e_1,[e_1[e_1,e_2]]] + 6[e_2,[e_1,e_2]] = 0$$
in $W_{\geq 1}$. For the proof of Proposition \ref{prop:derivations of Wn}, we first consider derivations of $W_{\geq n}$ of degree 0 and then find a relation in $W_{\geq n}$. It is not difficult to show that any derivation of $W_{\geq n}$ of degree 0 must be a multiple of $\ad_{e_0}$.

\begin{lem}\label{lem:derivations of Wn of degree 0}
    Letting $n \in \NN$, we have $\Der(W_{\geq n},W)_0 = \kk \ad_{e_0}$.
\end{lem}
\begin{proof}
    The result has already been shown for $n = 0,1$, so we may assume that $n \geq 2$.
    
    Let $d \in \Der(W_{\geq n},W)_0$. Then for all $k \geq n$ there exists $\lambda_k \in \kk$ such that $d(e_k) = \lambda_k e_k$. For all $\ell \geq n$, we have
    $$d([e_k,e_\ell]) = [e_k,d(e_\ell)] + [d(e_k),e_\ell] = (\lambda_k + \lambda_\ell)[e_k,e_\ell] = (\ell - k)(\lambda_k + \lambda_\ell)e_{k+\ell}.$$
    On the other hand,
    $$d([e_k,e_\ell]) = (\ell - k)d(e_{k+\ell}) = (\ell - k)\lambda_{k+\ell}e_{k+\ell}.$$
    Combining the above, we conclude that $\lambda_k + \lambda_\ell = \lambda_{k+\ell}$ for all $k \neq \ell$.
    
    Consider $2\lambda_k$ for some $k \geq n$. Let $\ell > 2k$. Then
    $$2\lambda_k + \lambda_\ell = \lambda_k + \lambda_{k+\ell} = \lambda_{2k+\ell} = \lambda_{2k} + \lambda_\ell.$$
    Therefore, $2\lambda_k = \lambda_{2k}$. Inductively, we deduce that $m\lambda_k = \lambda_{mk}$ for all $m \in \NN$.
    
    Now let $\mu = \lambda_n + \frac{\lambda_n}{n}$. Then
    $$n\mu = \lambda_n + n\lambda_n = \lambda_n + \lambda_{n^2} = \lambda_{n(n+1)} = n\lambda_{n+1},$$
    so $\mu = \lambda_{n+1}$. Now, we have $\lambda_n + \frac{\lambda_n}{n} = \lambda_{n+1}$, so
    $$\frac{\lambda_n}{n} = \lambda_{n+1} - \lambda_n,$$
    and thus $\lambda_n = n(\lambda_{n+1} - \lambda_n)$. Let $\lambda = \lambda_{n+1} - \lambda_n$. We claim that $d = \lambda \ad_{e_0}$.
    
    Then, for $m \geq n$,
    $$\lambda_m + \lambda = \lambda_m + \lambda_{n+1} - \lambda_n = \lambda_{n+m+1} - \lambda_n.$$
    But we also know that $\lambda_{n+m+1} = \lambda_{m+1} + \lambda_n$, so
    $$\lambda_m + \lambda = \lambda_{m+1} + \lambda_n - \lambda_n = \lambda_{m+1}.$$
    Since $\lambda_n = n\lambda$ and $\lambda_{m+1} = \lambda_m + \lambda$ for all $m \geq n$, it follows by induction that $\lambda_m = m\lambda$ for all $m \geq n$, proving the claim.
\end{proof}

We are nearly ready to prove Proposition \ref{prop:derivations of Wn}: we only need to find a relation in $W_{\geq n}$.

\begin{lem}\label{lem:relations}
    Let $n,m \in \ZZ$. We have the following relation on $e_n$ and $e_m$:
    \begin{multline*}
        nm[[e_n,e_m],[[e_n,e_m],[[e_n,e_m],[e_m,[e_n,[e_n,e_m]]]]]] \\
        + 3(m - n)(n + m)[[e_m,[e_n,[e_n,e_m]]],[[e_m,[e_n,[e_n,e_m]]],[e_n,e_m]]] = 0.
    \end{multline*}
\end{lem}
\begin{proof}
    We have
    \begin{align}
        [e_n,e_m] &= (m - n)e_{n+m},\label{eq:n+m bracket} \\
        [e_n,[e_n,e_m]] = (m - n)[e_n,e_{n+m}] &= (m - n)m e_{2n+m}, \nonumber \\
        [e_m,[e_n,[e_n,e_m]]] = (m - n)m [e_m,e_{2n+m}] &= 2(m - n)nm e_{2n+2m} \label{eq:2n+2m bracket}.
    \end{align}
    We can easily find a relation between $e_{n+m}$ and $e_{2n+2m}$:
    \begin{align}
        [e_{n+m},[e_{n+m},[e_{n+m},e_{2n+2m}]]] &= 6(n + m)^3 e_{5n+5m}, \label{eq:degree 5 relation part 1} \\
        [e_{2n+2m},[e_{2n+2m},e_{n+m}]] &= -(n + m)^2 e_{5n+5m}. \label{eq:degree 5 relation part 2}
    \end{align}
    Therefore, $[e_{n+m},[e_{n+m},[e_{n+m},e_{2n+2m}]]] + 6(n + m)[e_{2n+2m},[e_{2n+2m},e_{n+m}]] = 0$. We can now find a relation between $e_n$ and $e_m$ by substituting \eqref{eq:n+m bracket} and \eqref{eq:2n+2m bracket} into \eqref{eq:degree 5 relation part 1} and \eqref{eq:degree 5 relation part 2}:
    \begin{align*}
        [[e_n,e_m],[[e_n,e_m],[[e_n,e_m],[e_m,[e_n,[e_n,e_m]]]]]] &= 12nm(m - n)^4(n + m)^3 e_{5n+5m}, \\
        [[e_m,[e_n,[e_n,e_m]]],[[e_m,[e_n,[e_n,e_m]]],[e_n,e_m]]] &= -4n^2m^2(m - n)^3(n + m)^2 e_{5n+5m}.
    \end{align*}
    The lemma follows immediately.
\end{proof}

We can now prove Proposition \ref{prop:derivations of Wn}.

\begin{proof}[Proof of Proposition \ref{prop:derivations of Wn}]
    The result has already been shown for $n = 0,1$, so fix $n \geq 2$. By Lemma \ref{lem:derivations of Wn of degree 0}, we already know that $\Der(W_{\geq n},W)_0 = \kk \ad_{e_0}$. So, let $k \in \ZZ \nonzero$ and consider $d \in \Der(W_{\geq n},W)_k$. For $m \geq n$, write $d(e_m) = \lambda_m e_{m+k}$, where $\lambda_m \in \kk$. We claim that $d = \alpha \ad_{e_k}$ for some $\alpha \in \kk$, or in other words, that $\lambda_m = (m - k)\alpha$ for all $m \geq n$.
    
    Letting $r,s \geq n$ with $r \neq s$, Lemma \ref{lem:relations} implies that
    \begin{multline*}
        d\Bigl(rs[[e_r,e_s],[[e_r,e_s],[[e_r,e_s],[e_s,[e_r,[e_r,e_s]]]]]] \\
        + 3(s - r)(r + s)[[e_s,[e_r,[e_r,e_s]]],[[e_s,[e_r,[e_r,e_s]]],[e_r,e_s]]]\Bigr) = 0.
    \end{multline*}
    Therefore, applying the Leibniz rule repeatedly,
    \beq\label{eq:Leibniz on relation}
        \begin{split}
            rs\Bigl(&[[d(e_r),e_s],[[e_r,e_s],[[e_r,e_s],[e_s,[e_r,[e_r,e_s]]]]]] \\
            + &[[e_r,d(e_s)],[[e_r,e_s],[[e_r,e_s],[e_s,[e_r,[e_r,e_s]]]]]] \\
            + &[[e_r,e_s],[[d(e_r),e_s],[[e_r,e_s],[e_s,[e_r,[e_r,e_s]]]]]] \\
            + &[[e_r,e_s],[[e_r,d(e_s)],[[e_r,e_s],[e_s,[e_r,[e_r,e_s]]]]]] \\
            + &[[e_r,e_s],[[e_r,e_s],[[d(e_r),e_s],[e_s,[e_r,[e_r,e_s]]]]]] \\
            + &[[e_r,e_s],[[e_r,e_s],[[e_r,d(e_s)],[e_s,[e_r,[e_r,e_s]]]]]] \\
            + &[[e_r,e_s],[[e_r,e_s],[[e_r,e_s],[d(e_s),[e_r,[e_r,e_s]]]]]] \\
            + &[[e_r,e_s],[[e_r,e_s],[[e_r,e_s],[e_s,[d(e_r),[e_r,e_s]]]]]] \\
            + &[[e_r,e_s],[[e_r,e_s],[[e_r,e_s],[e_s,[e_r,[d(e_r),e_s]]]]]] \\
            + &[[e_r,e_s],[[e_r,e_s],[[e_r,e_s],[e_s,[e_r,[e_r,d(e_s)]]]]]]\Bigr) \\
            + 3(s - r)(r + s)\Bigl(&[[d(e_s),[e_r,[e_r,e_s]]],[[e_s,[e_r,[e_r,e_s]]],[e_r,e_s]]] \\
            + &[[e_s,[d(e_r),[e_r,e_s]]],[[e_s,[e_r,[e_r,e_s]]],[e_r,e_s]]] \\
            + &[[e_s,[e_r,[d(e_r),e_s]]],[[e_s,[e_r,[e_r,e_s]]],[e_r,e_s]]] \\
            + &[[e_s,[e_r,[e_r,d(e_s)]]],[[e_s,[e_r,[e_r,e_s]]],[e_r,e_s]]] \\
            + &[[e_s,[e_r,[e_r,e_s]]],[[d(e_s),[e_r,[e_r,e_s]]],[e_r,e_s]]] \\
            + &[[e_s,[e_r,[e_r,e_s]]],[[e_s,[d(e_r),[e_r,e_s]]],[e_r,e_s]]] \\
            + &[[e_s,[e_r,[e_r,e_s]]],[[e_s,[e_r,[d(e_r),e_s]]],[e_r,e_s]]] \\
            + &[[e_s,[e_r,[e_r,e_s]]],[[e_s,[e_r,[e_r,d(e_s)]]],[e_r,e_s]]] \\
            + &[[e_s,[e_r,[e_r,e_s]]],[[e_s,[e_r,[e_r,e_s]]],[d(e_r),e_s]]] \\
            + &[[e_s,[e_r,[e_r,e_s]]],[[e_s,[e_r,[e_r,e_s]]],[e_r,d(e_s)]]]\Bigr) = 0.
        \end{split}
    \eeq
    Substituting $d(e_r) = \lambda_r e_{r+k}$ and $d(e_s) = \lambda_s e_{s+k}$ into \eqref{eq:Leibniz on relation} and simplifying,
    \begin{multline*}
        (k^2 + ks + 6s^2 + kr + 12rs + 6r^2)(k^2 + ks - 2s^2 + kr - 2r^2) \\ 
        \times ((s - k)\lambda_r - (r - k)\lambda_s)k(s - r)^3rs e_{5r + 5s + k} = 0.
    \end{multline*}
    Now, $k(s - r)^3rs \neq 0$, since $r,s,k \neq 0$ and $r \neq s$. Hence,
    $$(k^2 + ks + 6s^2 + kr + 12rs + 6r^2)(k^2 + ks - 2s^2 + kr - 2r^2)((s - k)\lambda_r - (r - k)\lambda_s) = 0,$$
    for all $r,s \geq n$ with $r \neq s$. For brevity, let
    $$H_k(r,s) = (k^2 + ks + 6s^2 + kr + 12rs + 6r^2)(k^2 + ks - 2s^2 + kr - 2r^2).$$
    It follows that
    \beq\label{eq:lambdas}
        (s - k)\lambda_r = (r - k)\lambda_s
    \eeq
    for all $r,s \geq n$ such that $H_k(r,s) \neq 0$.

    Fix $m \geq n$ such that $m \neq k$. Letting $\alpha = \frac{\lambda_m}{m - k}$, we get $\lambda_\ell = (\ell - k)\alpha$ for all $\ell \geq n$ such that $H_k(m,\ell) \neq 0$, by \eqref{eq:lambdas}.
    
    Now suppose $\ell \geq n$ such that $H_k(m,\ell) = 0$. We want to show that we still have $\lambda_\ell = (\ell - k)\alpha$ in this case. Since $H_k(m,r)$ is non-constant in $r$, we can choose $r \geq n$ such that $r \neq k$, $H_k(m,r) \neq 0$ and $H_k(\ell,r) \neq 0$. Since $H_k(m,r) \neq 0$, we already know that $\lambda_r = (r - k)\alpha$. Furthermore, by \eqref{eq:lambdas}, we get
    $$(r - k)\lambda_\ell = (\ell - k)\lambda_r = (\ell - k)(r - k)\alpha,$$
    since $H_k(\ell,r) \neq 0$. It follows that $\lambda_\ell = (\ell - k)\alpha$ for all $\ell \geq n$, even if $H_k(m,\ell) = 0$. Therefore, $d = \alpha \ad_{e_k}$, as claimed.
\end{proof}

\section{Derivations of general submodule-subalgebras}\label{sec:general submodule-subalgebras}

We now consider derivations of general submodule-subalgebras of $\W$. It is standard that $\Ext$ groups of a filtered object are related to $\Ext$ groups of the associated graded object (see, for example, \cite[Proposition 2.6.10]{Bjork}). We will see that this is also the case in our situation: by an associated graded argument, we will be able to extract a lot of information about derivations of ungraded submodule-subalgebras from the results of Section \ref{sec:graded derivations}.

We begin with some definitions.

\begin{dfn}
    For an element $w = \sum_{i=m}^n \lambda_i e_i \in W$, where $n,m \in \ZZ$ with $n > m$, and $\lambda_i \in \kk$ with $\lambda_n \neq 0$, we call $n$ the \emph{degree} of $w$ and write $\deg(w) = n$. Furthermore, we call $\lambda_n e_n$ the \emph{leading term} of $w$ and write $\LT(w) = \lambda_n e_n$. Note that if $\deg(u) \neq \deg(v)$, we have $\deg([u,v]) = \deg(u) + \deg(v)$, where $u,v \in W$.
    
    For $f = \alpha \prod_{i=0}^n(t - \lambda_i)^{k_i} \in \kk[t]$, where $\alpha, \lambda_i \in \kk$, $\lambda_i \neq \lambda_j$ for $i \neq j$ and $k_i \geq 1$ for all $i$, the \emph{radical} of $f$ is $\rad(f) = \prod_{i=0}^n(t - \lambda_i)$. If $k_i = 1$ for all $i$, we say $f$ is \emph{reduced}. The \emph{vanishing locus} of $f$ is $V(f) = \{\lambda_1,\ldots,\lambda_n\}$.
    
    Let $f \in \kk(t) \nonzero$ and write $f = \frac{p}{q}$, where $p,q \in \kk[t] \nonzero$. Let $\xi \in \kk$ and write
    $$p = (t - \xi)^k r, \quad q = (t - \xi)^\ell s,$$
    where $k,\ell \in \NN$ and $r,s \in \kk[t]$ do not vanish at $\xi$. The \emph{order of vanishing of $f$ at $\xi$} is $\ord_\xi(f) = k - \ell$. If $\ord_\xi(f) < 0$, we say $f$ has a \emph{pole of order $-\ord_\xi(f)$}. By convention, $\ord_\xi(0) = \infty$ for all $\xi \in \kk$. Furthermore, for the derivation $f\del \in \Der(\kk(t))$, we simply write $\ord_\xi(f\del) = \ord_\xi(f)$.
\end{dfn}

The goal of this section is to compute $\Der(\W(f))$ for $f \in \kk[t] \nonzero$, thus generalising Proposition \ref{prop:derivations of Wn} to arbitrary submodule-subalgebras of $\W$.

\begin{thm}\label{thm:derivations of W(f)}
    Let $f \in \kk[t] \nonzero$. We have $\Der(\W(f)) = \{\ad_w \mid w \in \W(\rad(f))\}$. Consequently, $\Der(\W(f)) \cong \W(\rad(f))$, and
    $$\dim H^1(\W(f);\W(f)) = \deg(f) - \deg(\rad(f)).$$
\end{thm}

\begin{rem}
    A \emph{complete} Lie algebra is a Lie algebra with trivial center whose derivations are all inner. Theorem \ref{thm:derivations of W(f)} implies that $\W(f)$ is complete if and only if $f \in \kk[t]$ is reduced.
\end{rem}

In order to understand derivations of $\W(f)$, we will construct associated graded derivations, which are graded derivations of $W_{\geq n}$, where $n \coloneqq \deg(f\del)$. This will essentially be done by only considering the leading terms of elements of $\W(f)$. The following result is the first step in the construction of associated graded derivations.

\begin{lem}\label{lem:degree of derivation is bounded}
    Let $f \in \kk[t] \nonzero$. For $d \in \Der(\W(f),W)$, the set
    $$\{\deg(d(x)) - \deg(x) \mid x \in \W(f)\}$$
    is bounded above.
\end{lem}
\begin{proof}
    Let $d \in \Der(\W(f),W) \nonzero$. Let $n = \deg(f)$ and $x_i = t^{i-1}f\del$, so that $x_1,\ldots,x_n$ generate $\W(f)$ as a Lie algebra. Let
    $$N = \max\{\deg(d(x_i)) - \deg(x_i) \mid i \in \{1,\ldots,n\}\}.$$
    We claim that $\deg(d(x)) - \deg(x) \leq N$ for all $x \in \W(f)$. Note that
    $$d([x,y]) = [x,d(y)] + [d(x),y].$$
    Therefore, provided $\deg(d(x)) \leq \deg(x) + N$ and $\deg(d(y)) \leq \deg(y) + N$, we see that $\deg(d([x,y])) \leq \deg([x,y]) + N$. The claim now follows by an easy induction on $\deg(x)$.
\end{proof}

Lemma \ref{lem:degree of derivation is bounded} gives us a notion of degree for derivations $d \in \Der(\W(f),W)$.

\begin{dfn}
    Let $f \in \kk[t] \nonzero$. For $d \in \Der(\W(f),W) \nonzero$, the \emph{degree} of $d$ is
    $$\deg(d) = \max\{\deg(d(x)) - \deg(x) \mid x \in \W(f)\}.$$
    By convention, the degree of the zero derivation is $-\infty$.
\end{dfn}

We now consider leading terms of elements $x \in \W(f)$ such that $\deg(d(x)) = \deg(x) + \deg(d)$.

\begin{lem}\label{lem:d-compatible}
    Let $f \in \kk[t] \nonzero$, $d \in \Der(\W(f),W)$, and $x \in \W(f)$ such that $\deg(d(x)) = \deg(x) + \deg(d)$. Then $\deg(d(y)) = \deg(y) + \deg(d)$ for all $y \in \W(f)$ such that $\deg(y) = \deg(x)$. Furthermore, if $\LT(x) = \LT(y)$ then $\LT(d(x)) = \LT(d(y))$.
\end{lem}
\begin{proof}
    Let $y \in \W(f)$ such that $\deg(y) = \deg(x)$. Rescaling $y$ if necessary, we may assume that $\LT(y) = \LT(x)$. Let $z = x - y$, so $\deg(z) < \deg(x)$ and $d(y) = d(x) - d(z)$. By definition of $\deg(d)$, we know that
    $$\deg(d(z)) \leq \deg(z) + \deg(d) < \deg(x) + \deg(d) = \deg(d(x)).$$
    It follows that $\LT(d(y)) = \LT(d(x) - d(z)) = \LT(d(x))$ and that
    $$\deg(d(y)) = \deg(d(x)) = \deg(x) + \deg(d) = \deg(y) + \deg(d),$$
    as required.
\end{proof}

Given Lemma \ref{lem:d-compatible}, we can now make the following definition.

\begin{dfn}
    Let $f \in \kk[t] \nonzero$, $d \in \Der(\W(f),W)$, and $n = \deg(f\del)$. We say that $k \in \ZZ_{\geq n}$ is \emph{$d$-compatible} if $\deg(d(x)) = \deg(x) + \deg(d)$ for some (and therefore all) $x \in \W(f)$ such that $\deg(x) = k$.
\end{dfn}

\begin{ex}
    Let $f = t^2 + 1$ and let $d = \ad_{e_5 + e_3} \in \Der(\W(f))$. Then $\deg(d) = 5$. Therefore, 5 is not $d$-compatible, while every $n \in \ZZ_{\geq 1} \setminus \{5\}$ is $d$-compatible.
\end{ex}

We now complete the construction of associated graded derivations.

\begin{prop}\label{prop:associated graded derivation}
    Let $f \in \kk[t] \nonzero$ and $d \in \Der(\W(f),W)$. Define a $\kk$-linear map $\gr(d) \colon W_{\geq n} \to W$ by
    $$\gr(d)(e_k) = \begin{cases}
        \LT(d(x)), &\text{where } x \in \W(f) \text{ and }\LT(x) = e_k, \text{ if } k \text{ is } d\text{-compatible,} \\
        0, &\text{if } k \text{ is not } d\text{-compatible,}
    \end{cases}$$
    for $k \in \ZZ_{\geq n}$. Then $\gr(d) = \lambda \ad_{e_N} \in \Der(W_{\geq n},W)$ for some $\lambda \in \kk$ and $N \in \ZZ$.
\end{prop}
\begin{proof}
    First of all, by Lemma \ref{lem:d-compatible}, $\gr(d)$ is well-defined.
    
    If we prove that $\gr(d) \in \Der(W_{\geq n},W)$, then the result will follow from Proposition \ref{prop:derivations of Wn}. Therefore, it suffices to show that $\gr(d)$ is a derivation. In other words, it is enough to show that
    $$\gr(d)([e_k,e_\ell]) = [e_k,\gr(d)(e_\ell)] + [\gr(d)(e_k),e_\ell],$$
    for all $k,\ell \in \ZZ_{\geq n}$. Let $k,\ell \in \ZZ_{\geq n}$ be distinct, and let $u,v \in \W(f)$ such that $\LT(u) = e_k$ and $\LT(v) = e_\ell$. We have
    \beq\label{eq:Leibniz rule}
        d([u,v]) = [u,d(v)] + [d(u),v].
    \eeq
    The rest of the proof now follows by a straightforward case-by-case analysis, depending on whether $k$ and $\ell$ are $d$-compatible. We explicitly show the case where $k$ and $\ell$ are both $d$-compatible and leave the rest of the cases to the reader. Let $N = \deg(d)$. Since $\deg(d(u)) = k + N$, we have $\deg([d(u),v]) = k + \ell + N$ if $k + N \neq \ell$, but $\deg([d(u),v]) < k + \ell + N$ if $k + N = \ell$. We will therefore need to divide the proof into two cases.

    \vspace{2mm}

    \noindent \textbf{Case 1:} Either $k + N = \ell$ or $\ell + N = k$.

    \vspace{2mm}

    \noindent Without loss of generality, say $k + N = \ell$ and $\ell + N \neq k$. Then $k + \ell$ is $d$-compatible (which can be seen by looking at the degrees in \eqref{eq:Leibniz rule}), and $[\gr(d)(e_k),e_\ell] = [\LT(d(x)),e_\ell] = 0$, since $\deg(d(x)) = k + N = \ell$. Furthermore,
    $$\gr(d)([e_k,e_\ell]) = \LT(d([u,v])) = \LT([u,d(v)]) = [\LT(u),\LT(d(v))] = [e_k,\gr(d)(e_\ell)].$$
    Therefore, $\gr(d)([e_k,e_\ell]) = [e_k,\gr(d)(e_\ell)] + [\gr(d)(e_k),e_\ell]$.
    \vspace{2mm}

    \noindent \textbf{Case 2:} $k + N \neq \ell$ and $\ell + N \neq k$.

    \vspace{2mm}

    \noindent We have
    \begin{align*}
        [e_k,\gr(d)(e_\ell)] + [\gr(d)(e_k),e_\ell] &= [\LT(u),\LT(d(v))] + [\LT(d(u)),\LT(v)] \\
        &= \LT([u,d(v)]) + \LT([d(u),v]).
    \end{align*}
    Suppose $\LT([u,d(v)]) + \LT([d(u),v]) \neq 0$. Then $k + \ell$ is $d$-compatible and
    $$[e_k,\gr(d)(e_\ell)] + [\gr(d)(e_k),e_\ell] = \LT([u,d(v)] + [d(u),v]) = \LT(d([u,v])) = \gr(d)([e_k,e_\ell]).$$
    Otherwise, $\LT([u,d(v)]) + \LT([d(u),v]) = 0$, in which case $k + \ell$ is not $d$-compatible, so
    $$\gr(d)([e_k,e_\ell]) = 0 = [e_k,\gr(d)(e_\ell)] + [\gr(d)(e_k),e_\ell].$$
    This completes the proof.
\end{proof}

We now study $\Der(\W(f),\W)$. This is not strictly necessary to prove Theorem \ref{thm:derivations of W(f)}, which only considers $\Der(\W(f))$, but it is not much more difficult than considering $\Der(\W(f))$, and will be useful for the next section.

\begin{prop}\label{prop:derivations into W}
    Let $f \in \kk[t] \nonzero$ and $d \in \Der(\W(f),\W)$. Then $d = \ad_{\frac{h}{f}\del}$ for some $h \in \kk[t]$.
\end{prop}
\begin{proof}
    Consider $\gr(d) \in \Der(W_{\geq n},\W)$. By Proposition \ref{prop:associated graded derivation}, $\gr(d) = \lambda_0 \ad_{e_{k_0}}$ for some $\lambda_0 \in \kk$, $k_0 \in \ZZ$. Now let $d_1 = d - \lambda_0\ad_{e_{k_0}} \in \Der(\W(f),W)$ and note that $\deg(d_1) < \deg(d)$. Similarly to before, $\gr(d_1) = \lambda_1 \ad_{e_{k_1}}$ for some $\lambda_1 \in \kk$ and $k_1 \in \ZZ$, so we can consider $d_2 = d_1 - \lambda_1 \ad_{e_{k_1}} \in \Der(\W(f),W)$ with $\deg(d_2) < \deg(d_1)$. Continuing inductively, we get $d_1,d_2,d_3,\ldots \in \Der(\W(f),W)$ such that $\deg(d) > \deg(d_1) > \deg(d_2) > \deg(d_3) > \ldots$, and
    $$d = d_{m+1} + \sum_{i=0}^m \lambda_i \ad_{e_{k_i}},$$
    for all $m \in \NN$, where we set $\lambda_i = 0$ if $d_i = 0$. Let $g = \sum_{i=0}^\infty \lambda_i t^{k_i + 1} \in \kk((t^{-1}))$ and consider $\ad_{g\del} \in \Der(\kk((t^{-1}))\del)$, so that $d = \restr{\ad_{g\del}}{\W(f)}$. In other words, $[g\del,hf\del] \in \W$ for all $h \in \kk[t]$. Hence, $[g\del,f\del] \in \W$, so $f'g - fg' \in \kk[t]$, and thus
    $$tf'g - tfg' \in \kk[t].$$
    Similarly, $[g\del,tf\del] \in \W$, so
    $$fg + tf'g - tfg' \in \kk[t].$$
    Combining the above, we conclude that $h \coloneqq fg \in \kk[t]$, so $g = \frac{h}{f} \in \kk(t)$.
\end{proof}

If $d \in \Der(\W(f))$, then Proposition \ref{prop:derivations into W} tells us that $d = \ad_{g\del}$ for some $g \in \kk(t)$. For the proof of Theorem \ref{thm:derivations of W(f)}, we need to show that $g \in \rad(f)\kk[t]$, in other words, that $g \in \kk[t]$ and that $g$ vanishes at all the roots of $f$. In order to achieve this, we will consider orders of vanishing of $[f\del,g\del]$ at roots of $f$. The following easy lemma computes $\ord_\xi([f\del,g\del])$ for $\xi \in \kk$.

\begin{lem}\label{lem:order of vanishing}
    Let $\xi \in \kk$, $f \in \kk[t] \nonzero$, and $g \in \kk(t) \nonzero$. Then
    $$\ord_\xi([f\del,g\del]) = \ord_\xi(f) + \ord_\xi(g) - 1$$
    if $\ord_\xi(g) \neq \ord_\xi(f)$ and $\ord_\xi([f\del,g\del]) \geq 2\ord_\xi(f)$ if $\ord_\xi(g) = \ord_\xi(f)$.
\end{lem}
\begin{proof}
    Write $f = (t - \xi)^k p$ and $g = (t - \xi)^\ell q$, where $p \in \kk[t], q \in \kk(t)$ and $\ord_\xi(p) = \ord_\xi(q) = 0$, so that $k = \ord_\xi(f)$ and $\ell = \ord_\xi(g)$. Differentiating $f$ and $g$, we get
    $$f' = (t - \xi)^{k-1}(kp + (t - \xi)p'), \quad g' = (t - \xi)^{\ell-1}(\ell q + (t - \xi)q').$$
    Therefore,
    $$[f\del,g\del] = (fg' - f'g)\del = (t - \xi)^{k+\ell-1}((\ell - k)pq + (t - \xi)(pq' - p'q)).$$
    Let $r = (\ell - k)pq + (t - \xi)(pq' - p'q)$. It is clear that $\ord_\xi(r) = 0$ if $k \neq \ell$. Therefore,
    $$\ord_\xi(fg' - f'g) = \ord_\xi((t - \xi)^{k+\ell-1}r) = k + \ell - 1$$
    if $k \neq \ell$. If $k = \ell$, then $\ord_\xi(r) \geq 1$, so
    $$\ord_\xi(fg' - f'g) = \ord_\xi((t - \xi)^{2k-1}r) \geq 2k,$$
    which concludes the proof.
\end{proof}

We are now ready to prove Theorem \ref{thm:derivations of W(f)}.

\begin{proof}[Proof of Theorem \ref{thm:derivations of W(f)}]
    If $w \in \W(\rad(f))$, then by Lemma \ref{lem:order of vanishing}, it follows that $\ad_w \in \Der(\W(f))$.
    
    Let $d \in \Der(\W(f))$. By Proposition \ref{prop:derivations into W}, we know that $d = \ad_{\frac{h}{f}\del}$ for some $h \in \kk[t]$. Let $g = \frac{h}{f}$. Since $[g\del,f\del] \in \W(f)$, $f$ divides $f'g - fg'$. Therefore,
    $$\ord_\xi(f'g - fg') \geq \ord_\xi(f)$$
    for all $\xi \in V(f)$. Assume, for a contradiction, that $\ord_\xi(g) \leq 0$ for some $\xi \in V(f)$. Then, by Lemma \ref{lem:order of vanishing},
    $$\ord_\xi(f) \leq \ord_\xi(f'g - fg') = \ord_\xi(f) + \ord_\xi(g) - 1 \leq \ord_\xi(f) - 1,$$
    a contradiction. Therefore, $\ord_\xi(g) \geq 1$ for all $\xi \in V(f)$. It follows that $g \in \kk[t]$ (since $g$ can only have poles at roots of $f$), and that $\rad(f)$ divides $g$, so $g\del \in \W(\rad(f))$.
\end{proof}

\section{Extensions of submodule-subalgebras}\label{sec:extensions}

Having computed $\Ext_{U(\W(f))}^1(M,\W(f))$ for $M = \kk$, we move on to the computation of the $\Ext$ group when $M$ is an arbitrary one-dimensional representation of $\W(f)$. We may view an element of $\Ext_{U(\W(f))}^1(M,\W(f))$ as a non-split short exact sequence of $U(\W(f))$-modules
\beq\label{eq:SES one-dim extension}
    0 \to \W(f) \to X \to M \to 0.
\eeq
Although it is not immediately obvious, this computation is closely related to the results from previous sections on derivations. This is because $M$ is trivial as a representation of $\W(f^2) = [\W(f),\W(f)]$ (see Lemma \ref{lem:derived subalgebra}). Therefore, regarding \eqref{eq:SES one-dim extension} as a sequence of $U(\W(f^2))$-modules, we get an element of
$$\Ext_{U(\W(f^2))}^1(\kk,\W(f)) \cong \Der(\W(f^2),\W(f))/\Inn(\W(f^2),\W(f))$$
(recall the isomorphism $\Ext_{U(L)}^1(\kk,M) \cong \Der(L,M)/\Inn(L,M)$ from the beginning of Section \ref{sec:graded derivations}). We can use Proposition \ref{prop:derivations into W} to compute $\Der(\W(f^2),\W(f))$, which then allows us to determine the structure of $X$ as a representation of $\W(f^2)$. This will give a lot of information into the structure of $X$ as a representation of $\W(f)$.

First of all, it is straightforward to show that $X$ has the structure of a Lie algebra in a natural way.

\begin{lem}\label{lem:SES gives Lie algebra}
    Let $0 \to L \to \overline{L} \xrightarrow{\pi} M \to 0$ be a short exact sequence of $U(L)$-modules, where $L$ is a Lie algebra and $\dim(M) = 1$. Writing $\overline{L} = L \oplus \kk x$ as a vector space (where $\pi(\kk x) = M$), there is a unique Lie bracket on $\overline{L}$ extending the action of $L$ on $\overline{L}$. This is given by $[u,v]_{\overline{L}} = [u,v]_L$ and $[u,x]_{\overline{L}} = u \cdot x$ for $u,v \in L$.
\end{lem}
\begin{proof}
    It suffices to check that
    \beq\label{eq:Jacobi for one-dim extension}
        [u,[v,x]] + [v,[x,u]] + [x,[u,v]] = 0
    \eeq
    for $u,v \in L$. We have
    $$[[u,v],x] = [u,v] \cdot x = u \cdot (v \cdot x) - v \cdot (u \cdot x),$$
    since $\overline{L}$ is a $U(L)$-module. Therefore,
    $$[[u,v],x] = [u,[v,x]] - [v,[u,x]],$$
    so \eqref{eq:Jacobi for one-dim extension} holds.
\end{proof}

Lemma \ref{lem:SES gives Lie algebra} says that studying the group $\Ext_{U(\W(f))}^1(\W(f),M)$, where $M$ is a one-dimensional $U(\W(f))$-module, is essentially the same as studying Lie algebras $\overline{L}$ containing $\W(f)$ such that $\dim(\overline{L}/\W(f)) = 1$. Thus, we make the following definition.

\begin{dfn}
    Let $L$ be a Lie algebra. A \emph{one-dimensional extension} of $L$ is a Lie algebra $\overline{L}$ containing $L$ such that $\dim(\overline{L}/L) = 1$ and the short exact sequence of $U(L)$-modules
    $$0 \to L \to \overline{L} \to \overline{L}/L \to 0$$
    is non-split.
\end{dfn}

We therefore consider one-dimensional extensions of $\W(f)$, where $f \in \kk[t] \nonzero$. There are two types of ``obvious" one-dimensional extensions of $\W(f)$. Firstly, suppose $f$ is non-reduced, so that there exists $w \in \W(\rad(f)) \setminus \W(f)$. Since $\W(f)$ is an ideal of $\W(\rad(f))$, it follows that $\overline{L} \coloneqq \W(f) \oplus \kk w$ is a one-dimensional extension of $\W(f)$. Note that $\overline{L}/\W(f)$ is trivial as a representation of $\W(f)$, since $\W(f)$ is an ideal of $\overline{L}$. Hence, this construction corresponds to an element of $\Ext_{U(\W(f))}^1(\kk,\W(f))$.

The other ``obvious" one-dimensional extensions come from removing non-repeated roots of $f$. In other words, if 
$\xi \in V(f)$ with $\ord_\xi(f) = 1$, then $\W(\frac{f}{t - \xi})$ is a one-dimensional extension of $\W(f)$ which is not contained in $\W(\rad(f))$. In this case, the quotient $\W(\frac{f}{t - \xi})/\W(f)$ is isomorphic to $\W/\W(t - \xi)$ as a $U(\W(f))$-module, where the module structure on $\W/\W(t - \xi)$ is induced from $\W(f)$ acting on $\W$ and $\W(t - \xi)$, noting that $\W(f)$ is a subalgebra of both $\W$ and $\W(t - \xi)$. Therefore, this corresponds to an element of
$$\Ext_{U(\W(f))}^1(\W/\W(t - \xi),\W(f)).$$
The main goal of this section is to prove that all one-dimensional extensions of $\W(f)$ are one of these two types.

\begin{thm}\label{thm:one-dimensional extension}
    Let $f \in \kk[t] \nonzero$ and let $M$ be a one-dimensional $U(\W(f))$-module. If $M \cong \W/\W(t - \xi)$ for some $\xi \in V(f)$ with $\ord_\xi(f) = 1$, then
    $$\dim(\Ext_{U(\W(f))}^1(\W/\W(t - \xi),\W(f))) = 1.$$
    In this case, the unique non-split extension of $\W/\W(t - \xi)$ by $\W(f)$ is $\W(\frac{f}{t - \xi})$. If $M \cong \kk$ is trivial, we have
    $$\dim(\Ext_{U(\W(f))}^1(\kk,\W(f))) = \deg(f) - \deg(\rad(f)).$$
    Otherwise, $\Ext_{U(\W(f))}^1(M,\W(f)) = 0$.
\end{thm}

The first step in proving Theorem \ref{thm:one-dimensional extension} is computing the derived subalgebra of $\W(f)$. This will be used to deduce that any one-dimensional representation of $\W(f)$ is trivial when viewed as a representation of $\W(f^2)$.

\begin{lem}\label{lem:derived subalgebra}
    For $f \in \kk[t]$, the derived subalgebra of $\W(f)$ is $\W(f^2)$.
\end{lem}
\begin{proof}
    The result follows by noting that $[gf\del,hf\del] = (gh' - g'h)f^2\del$ for $g,h \in \kk[t]$.
\end{proof}

\begin{rem}\label{rem:residually nilpotent}
    For a Lie algebra $L$, write $L_{(0)} = L$ and $L_{(n+1)} = [L,L_{(n)}]$ for the lower central series of $L$, where $n \in \NN$. We say that a Lie algebra $L$ is \emph{residually nilpotent} if $\bigcap_{n \in \NN} L_{(n)} = 0$.

    Since we computed $\W(f)_{(1)}$ in Lemma \ref{lem:derived subalgebra}, one might ask about further terms in the lower central series of $\W(f)$. It is not difficult to show that
    $$\W(f)_{(2)} = \W(\gcd(f',f)f^2)$$
    for $f \in \kk[t]$. Note that $\gcd(f',f) = 1$ if and only if $f$ is reduced, so we have $\W(f)_{(2)} = \W(f^2) = \W(f)_{(1)}$ and therefore $\W(f)_{(n)} = \W(f^2)$ for all $n \geq 1$. On the other hand, when $f \in \kk[t]$ is not reduced, one can easily show that $\W(f)$ is residually nilpotent.
\end{rem}

For the next step, we will need a result from \cite{PetukhovSierra}. Note that the result in \cite{PetukhovSierra} considers subalgebras of $W$, but a similar result is true for subalgebras of $\W$ with a nearly identical proof.

\begin{prop}[{\cite[Proposition 3.2.7]{PetukhovSierra}}]\label{prop:Alexey}
    Let $L$ be a subalgebra of $\W$ of finite codimension. Then there exists a monic polynomial $f \in \kk[t] \nonzero$ such that
    $$\W(f) \subseteq L \subseteq \W(\rad(f)).$$
    Furthermore, if we assume $f$ is of minimal degree, then such $f$ is unique.
\end{prop}

The following result implies that all one-dimensional extensions of $\W(f)$ are contained in $\W$. Note that the proposition considers arbitrary subalgebras of $\W$ of finite codimension, which is more difficult than restricting to submodule-subalgebras, but will be useful for Sections \ref{sec:isomorphism} and \ref{sec:universal}.

\begin{prop}\label{prop:extensions are in W}
    Let $L$ be a Lie subalgebra of $\W$ of finite codimension, and let $\overline{L}$ be a one-dimensional extension of $L$. Then $\overline{L}$ can be uniquely embedded in $\W$ such that the diagram
    \begin{center}
        \begin{tikzcd}
            L \arrow[d, hook] \arrow[r, hook] & \W \\
            \overline{L} \arrow[ru, hook]     &   
        \end{tikzcd}
    \end{center}
    commutes.
\end{prop}
\begin{proof}
    By Proposition \ref{prop:Alexey}, there exists $f \in \kk[t] \nonzero$ such that
    $$\W(f) \subseteq L \subseteq \W(\rad(f)),$$
    since $L$ has finite codimension in $\W$. Write $\overline{L} = L \oplus \kk x$ as a vector space. Since $\overline{L}/L$ is a one-dimensional representation of $L$, it follows that $\overline{L}/L$ is trivial as a representation of $[L,L]$. Hence, $[x,[L,L]] \subseteq L$. Since $\W(f) \subseteq L$, we know that $\W(f^2) \subseteq [L,L]$. Therefore,
    $$[x,\W(f^2)] \subseteq L \subseteq \W(\rad(f)).$$
    Hence, $\restr{\ad_x}{\W(f^2)} \in \Der(\W(f^2),\W(\rad(f)))$. By Proposition \ref{prop:derivations into W}, we have $\ad_x = \ad_{\frac{h}{f^2}\del}$ on $\W(f^2)$, for some $h \in \kk[t]$. Let $g = \frac{h}{f^2}$, so that $[w,x] = [w,g\del]$ for all $w \in \W(f^2)$. Note that $\W(f)$ and $\W(f^2)$ are ideals of $L$, since $L \subseteq \W(\rad(f))$.

    Let $u \in L \nonzero, v \in \W(f^2) \nonzero$. We have
    \beq\label{eq:Jacobi on extension}
        [v,[u,x]] = [u,[v,x]] - [[u,v],x] = [u,[v,g\del]] - [[u,v],g\del] = [v,[u,g\del]],
    \eeq
    where we used that $[u,v] \in \W(f^2)$ in the second equality.
    
    Let $M = \overline{L}/L$ as a $U(L)$-module. We can regard the short exact sequence
    \beq\label{eq:SES with L}
        0 \to L \to \overline{L} \to M \to 0
    \eeq
    as an element of $\Ext_{U(L)}^1(M,L)$ and consider the map
    $$\varphi \colon \Ext_{U(L)}^1(M,L) \to \Ext_{U(L)}^1(M,\kk(t)\del),$$
    where we view $L$ as a subalgebra of $\kk(t)\del$. Under the map $\varphi$, the short exact sequence \eqref{eq:SES with L} gets mapped to
    \beq\label{eq:SES with X}
        0 \to \kk(t)\del \to X \to M \to 0,
    \eeq
    where $X$ is the $U(L)$-module $\kk(t)\del \oplus \kk x$ with $L$ acting by its adjoint action on $\kk(t)\del$ and the action of $L$ on $x$ agreeing with the bracket on $\overline{L}$. We claim that \eqref{eq:SES with X} splits, in other words, that the image of \eqref{eq:SES with L} under $\varphi$ is zero. This will then allow us to conclude that $\overline{L}$ can be embedded into $\kk(t)\del$.
    
    Let $y = x - g\del \in X$. By \eqref{eq:Jacobi on extension}, we have $[v,[u,y]] = 0$ for all $u \in L \nonzero$, $v \in \W(f^2) \nonzero$. There exist $\kk$-linear maps $a \colon L \to \kk(t)\del$ and $\lambda \colon L \to \kk$ such that
    $$[u,y] = a(u) + \lambda(u)x$$
    for all $u \in L$. We can also write
    $$[u,x] = b(u) + \lambda(u)x$$
    for all $u \in L$, where $b(u) = a(u) + [u,g\del] \in L$. Therefore,
    $$0 = [v,[u,y]] = [v,a(u) + \lambda(u)x] = [v,a(u) + \lambda(u)g\del],$$
    where in the final equality we used that $[v,x] = [v,g\del]$ for all $v \in \W(f^2)$. It follows that $a(u) + \lambda(u)g\del \in \kk(t)\del$ is a scalar multiple of $v$, since the centraliser of any nonzero element of $\kk(t)\del$ is one-dimensional. But this must hold for all $v \in \W(f^2) \nonzero$, so $a(u) + \lambda(u)g\del = 0$. Now, we have
    \beq\label{eq:SES with X splits}
        [u,y] = a(u) + \lambda(u)x = a(u) + \lambda(u)x - (a(u) + \lambda(u)g\del) = \lambda(u)y,
    \eeq
    for all $u \in L$. Letting $\overline{x}$ be the image of $x$ in $M$, the map
    \begin{align*}
        s \colon M &\to X \\
        \overline{x} &\mapsto y
    \end{align*}
    is a section of \eqref{eq:SES with X}, so \eqref{eq:SES with X} splits, as claimed.

    Suppose $g\del \in L$. Then $y \in \overline{L}$, so we could take the codomain of $s$ to be $\overline{L}$. In this case, $s \colon M \to \overline{L}$ is a section of \eqref{eq:SES with L}. However, we assumed that \eqref{eq:SES with L} is non-split, so this is a contradiction. Hence, we must have $g\del \not\in L$, and therefore $\overline{L} \cong L \oplus \kk g\del$. Without loss of generality, we make the identification $\overline{L} = L \oplus g\del$, so we view $\overline{L}$ as a subalgebra of $\kk(t)\del$.
    
    We claim that $\ord_\xi(g) \geq 0$ for all $\xi \in V(f)$. Since $g = \frac{h}{f^2}$ can only have poles at roots of $f$, this will imply that $g \in \kk[t]$. In other words, this will show that $\overline{L} \subseteq \W$, which is enough to finish the proof.

    Take $\xi \in V(f)$ and suppose $\ord_\xi(g) < 0$. Define sets $S = \{\ord_\xi(w) \mid w \in L\}$ and $\overline{S} = \{\ord_\xi(w) \mid w \in \overline{L}\}$. Let $n_0 = \min(S)$, and let $w_0 \in L$ such that $\ord_\xi(w_0) = n_0$. If $w = u + \alpha g\del \in \overline{L}$, where $u \in L$ and $\alpha \in \kk$, then either $\ord_\xi(w) = \ord_\xi(g)$ if $\alpha \neq 0$, or $\ord_\xi(w) \geq \ord_\xi(w_0)$ if $\alpha = 0$. Therefore, $\overline{S} = S \cup \{\ord_\xi(g)\}$. Furthermore, since $\W(f) \subseteq L$, we know that $S$ contains $\ZZ_{\geq \ord_\xi(f)}$.
    
    Since $[w_0,g\del] \in \overline{L}$, we know that $\ord_\xi([w_0,g\del]) \in \overline{S} = S \cup \{\ord_\xi(g)\}$. By Lemma \ref{lem:order of vanishing}, we have
    $$\ord_\xi([w_0,g\del]) = \ord_\xi(w_0) + \ord_\xi(g) - 1 < n_0.$$
    By minimality of $n_0$, $\ord_\xi([w_0,g\del]) \not\in S$, so we must have that $\ord_\xi([w_0,g\del]) = \ord_\xi(g)$. Therefore, $n_0 = \ord_\xi(w_0) = 1$.

    Let $n = 1 - \ord_\xi(g) \in \NN$. Then $\ord_\xi([w,g\del]) = \ord_\xi(w) - n$ for all $w \in L$, by Lemma \ref{lem:order of vanishing}. Let $m \in S$ be such that $m \geq 2$, and let $w \in L$ such that $\ord_\xi(w) = m$. Then
    $$\ord_\xi([w,g\del]) = m - n \geq 2 - n = \ord_\xi(g) + 1.$$
    Since $\ord_\xi([w,g\del]) \in \overline{S} = S \cup \{\ord_\xi(g)\}$ and $\ord_\xi([w,g\del]) \neq \ord_\xi(g)$, we must have that $\ord_\xi([w,g\del]) \in S$. Hence, $m - n \in S$ for all $m \in S$ such that $m \geq 2$. In particular, this means that if $m \in S$ such that $m \geq 2$, then $m$ must be larger than $n$, since $\min(S) = 1$. Noting that $n = 1 - \ord_\xi(g) \geq 2$, we therefore see that $2 \not\in S$. 
    
    Applying the above inductively, we deduce that $m - kn \in S$ for all $m \in S, k \in \NN$ such that $m > kn$. Since $2 \not\in S$, it follows that $S$ does not contain any elements which are 2 modulo $n$. However, this contradicts $\ZZ_{\geq \ord_\xi(f)} \subseteq S$. It follows that $\ord_\xi(g) \geq 0$ for all $\xi \in V(f)$, as claimed.
\end{proof}

A special case of Proposition \ref{prop:extensions are in W} is a description of derivations of arbitrary subalgebras of $\W$ of finite codimension.

\begin{cor}\label{cor:derivations}
    Let $L$ be a subalgebra of $\W$ of finite codimension. Then
    $$\Der(L) = \{\ad_w \mid w \in \W \text{ such that } \ad_w(L) \subseteq L\}.$$
    In other words, derivations of $L$ are restrictions of derivations of $\W$. Consequently, $H^1(L;L) \cong N_{\W}(L)/L$, where $N_{\W}(L)$ is the normaliser of $L$ in $\W$. \qed
\end{cor}

Having shown that all one-dimensional extensions of $\W(f)$ are contained in $\W$, we can now explicitly compute them to prove Theorem \ref{thm:one-dimensional extension}.

\begin{proof}[Proof of Theorem \ref{thm:one-dimensional extension}]
    If $M \cong \kk$ then the result follows by Theorem \ref{thm:derivations of W(f)}. So, assume $M$ is nontrivial as a representation of $\W(f)$.

    Let $0 \to \W(f) \to \overline{L} \to M \to 0$ be a non-split short exact sequence of representations of $\W(f)$. By Proposition \ref{prop:extensions are in W}, we may assume that $\overline{L} \subseteq \W$. It suffices to show that $\overline{L} = \W(\frac{f}{t - \xi})$ for some $\xi \in V(f)$ with $\ord_\xi(f) = 1$.

    Write $\overline{L} = \W(f) \oplus \kk g\del$ as a vector space, where $g \in \kk[t] \nonzero$. If $g\del \in \W(\rad(f))$, then $\overline{L} \subseteq \W(\rad(f))$. So, suppose $g\del \not\in \W(\rad(f))$, in other words, there exists $\xi \in V(f)$ such that $g(\xi) \neq 0$. Note that, if $\lambda \in \kk \nonzero$ and $p \in \kk[t]$, then $\ord_\xi(pf + \lambda g) = 0$, while $\ord_\xi(pf) \geq \ord_\xi(f)$. It follows that if $w \in \overline{L}$, then $\ord_\xi(w) = 0$ or $\ord_\xi(w) \geq \ord_\xi(f)$. Note that $[f\del,g\del] \in \overline{L} \nonzero$, and
    $$\ord_\xi([f\del,g\del]) = \ord_\xi(f) + \ord_\xi(g) - 1 = \ord_\xi(f) - 1,$$
    by Lemma \ref{lem:order of vanishing}. It follows that $\ord_\xi([f\del,g\del]) = 0$, so $\ord_\xi(f) = 1$.

    For $k \in \NN$, write $[t^kf\del,g\del] = w_k + \lambda_k g\del$, where $w_k \in \W(f)$ and $\lambda_k \in \kk$. Note that
    $$[t^kf\del,g\del] = (t^kfg' - kt^{k-1}fg - t^kf'g)\del.$$
    Therefore, $w_k = (t^kfg' - kt^{k-1}fg - t^kf'g)\del - \lambda_k g\del$. Evaluating $w_k$ at $\xi$, we get
    $$-\xi^kf'(\xi)g(\xi) - \lambda_k g(\xi) = 0,$$
    since $w_k \in \W(f)$ vanishes at $\xi$. Hence, $\lambda_k = -\xi^k f'(\xi)$ for all $k \in \NN$, since $g(\xi) \neq 0$.

    The only assumptions we made on $\xi$ were that $\xi \in V(f)$ and $g(\xi) \neq 0$. There is thus only one such $\xi$, and in particular $g(\xi') = 0$ for all $\xi' \in V(f) \setminus \{\xi\}$.

    To end the proof, we claim that $g\del \in \W(\frac{f}{t - \xi})$, so that $\overline{L} \subseteq \W(\frac{f}{t - \xi})$. In particular, since $\W(\frac{f}{t - \xi})$ is a one-dimensional extension of $\W(f)$, this will immediately imply that $\overline{L} = \W(\frac{f}{t - \xi})$. In order to prove that $g\del \in \W(\frac{f}{t - \xi})$, it suffices to show that $\ord_\mu(g) \geq \ord_\mu(f)$ for all $\mu \in V(f) \setminus \{\xi\}$.
    
    We have already shown that $\ord_\mu(g) \geq 1$ for all $\mu \in V(f) \setminus \{\xi\}$, so if $\ord_\mu(f) = 1$ then $\ord_\mu(g) \geq \ord_\mu(f)$. So, suppose $\ord_\mu(f) \geq 2$ and $\ord_\mu(g) < \ord_\mu(f)$ for some $\mu \in V(f) \setminus \{\xi\}$. Recall that
    $$[f\del,g\del] = w_0 + \lambda_0 g\del,$$
    where $w_0 \in \W(f)$ and $\lambda_0 = -f'(\xi) \neq 0$. Therefore, $\ord_\mu([f\del,g\del]) = \ord_\mu(g)$. On the other hand, Lemma \ref{lem:order of vanishing} says that
    $$\ord_\mu([f\del,g\del]) = \ord_\mu(f) + \ord_\mu(g) - 1 \geq \ord_\mu(g) + 1,$$
    since $\ord_\mu(f) \geq 2$, a contradiction. Hence, $\ord_\mu(g) \geq \ord_\mu(f)$ for all $\mu \in V(f) \setminus \{\xi\}$, as claimed.
\end{proof}

The following corollary, which follows immediately from Theorem \ref{thm:one-dimensional extension}, summarises the two types of one-dimensional extensions which $\W(f)$ has.

\begin{cor}
    Suppose $\overline{L}$ is a one-dimensional extension of $\W(f)$, where $f \in \kk[t] \nonzero$. Then either $\overline{L} \subseteq \W(\rad(f))$ or $\overline{L} = \W(\frac{f}{t - \xi})$, where $\xi \in V(f)$ with $\ord_\xi(f) = 1$. \qed
\end{cor}

\section{Isomorphisms between subalgebras of finite codimension}\label{sec:isomorphism}

As an application of our results on extensions of subalgebras of $\W$, we study isomorphisms between subalgebras of $\W$ of finite codimension. We prove that any isomorphism between subalgebras of finite codimension extends to an automorphism of $\W$, which allows us to classify submodule-subalgebras of $\W$ up to isomorphism and compute automorphism groups of submodule-subalgebras. This will be done by an inductive argument: if we have an isomorphism $\varphi \colon L_1 \to L_2$ between subalgebras of $\W$, we will show that $\varphi$ extends to an isomorphism between one-dimensional extensions of $L_1$ and $L_2$.

\begin{thm}\label{thm:extending isomorphisms}
    Let $L_1$ and $L_2$ be subalgebras of $\W$ of finite codimension and suppose there is an isomorphism of Lie algebras $\varphi \colon L_1 \to L_2$. Then $\varphi$ extends to an automorphism of $\W$.
\end{thm}

In particular, if $f \in \kk[t] \nonzero$, then Theorem \ref{thm:extending isomorphisms} implies that automorphisms of $\W(f)$ are restrictions of automorphisms of $\W$. This is similar to Theorem \ref{thm:derivations of W(f)}, which says that derivations of $\W(f)$ are restrictions of derivations of $\W$.

Theorem \ref{thm:extending isomorphisms} also implies that if $L_1$ and $L_2$ are isomorphic subalgebras of $\W$ of finite codimension, then $\codim_{\W}(L_1) = \codim_{\W}(L_2)$, which is not obvious a priori.

\begin{rem}
    Theorem \ref{thm:extending isomorphisms} is not true for arbitrary infinite-dimensional subalgebras of $\W$. Taking $f \in \kk[t]$ of degree at least 2, consider the Lie algebra $L(f)$ from \cite{Buzaglo} (see Notation \ref{ntt:L(f,g)}). In particular, $L(f)$ has infinite codimension in $\W$, and there exists $h \in \kk[t]$ such that $L(f) \cong \W(h)$, by \cite[Lemma 4.12]{Buzaglo}. However, this isomorphism clearly does not extend to an automorphism of $\W$. In fact, it does not even extend to an endomorphism of $\W$, since any nonzero endomorphism of $\W$ is an automorphism \cite{Du}.
\end{rem}

If we have a submodule-subalgebra $\W(f)$, where $f \in \kk[t]$ is non-constant, then we know that $\W(f)$ has one-dimensional extensions (for example, we can take $\W(\frac{f}{t - \xi})$, where $f(\xi) = 0$). However, it is not immediately clear whether an arbitrary subalgebra of $\W$ of finite codimension has any one-dimensional extensions. Our next goal is to show that any such subalgebra has one-dimensional extensions. We will requite notation for the derived series of a Lie algebra.

\begin{ntt}
    For a Lie algebra $L$, let $D(L) = [L,L]$. We write $D^0(L) = L$ and 
    $$D^{n+1}(L) = D(D^{n}(L)) = [D^n(L),D^n(L)]$$ for the derived series of $L$, where $n \in \NN$.
\end{ntt}

\begin{rem}
    Analogously to residual nilpotence, we say a Lie algebra $L$ is \emph{residually solvable} if $\bigcap_{n \in \NN} D^n(L) = 0$. By Lemma \ref{lem:derived subalgebra}, we can see that if $f \in \kk[t]$ is non-constant, then $\W(f)$ is residually solvable.

    Note that residual solvability is, in some sense, a weak property. For example, a free Lie algebra is residually solvable. On the other hand, $\W$ is not residually solvable, since $D^n(\W) = \W$ for all $n \in \NN$.
\end{rem}

In fact, we will see that $\W(f)$ satisfies a stronger property defined below, provided $f \in \kk[t]$ is non-constant.

\begin{dfn}
    We say that a Lie algebra $L$ is \emph{strongly residually solvable} if $L$ is residually solvable and every finite-dimensional quotient of $L$ is solvable.
\end{dfn}

Given a subalgebra $L \subseteq \W$ of finite codimension, the next result shows that $L$ is strongly residually solvable. Lie's theorem will then allow us to deduce that all finite-dimensional irreducible representations of $L$ are one-dimensional.

\begin{lem}\label{lem:L/I is solvable}
    If $L$ is a proper subalgebra of $\W$ of finite codimension then $L$ is strongly residually solvable.
\end{lem}
\begin{proof}
    Let $I$ be an ideal of $L$ of finite codimension. By Proposition \ref{prop:Alexey}, there exist $f,g \in \kk[t] \nonzero$ such that
    $$\W(f) \subseteq L \subseteq \W(\rad(f)), \quad \W(g) \subseteq I \subseteq \W(\rad(g)).$$
    Since $\W(\rad(f))$ is residually solvable, it follows that $L$ is also residually solvable. It remains to show that $L/I$ is solvable.
    
    Note that $\W(g) \subseteq I \subseteq L \subseteq \W(\rad(f))$, so $\rad(f)$ divides $\rad(g)$. Suppose $\rad(g) \neq \rad(f)$. Then there exists $\xi \in V(g) \setminus V(f)$. Let $w \in I \nonzero$ such that $\ord_\xi(w)$ is minimal. Then $[f\del,w] \in I$, since $f\del \in L$ and $I$ is an ideal of $L$. By Lemma \ref{lem:order of vanishing},
    $$\ord_\xi([f\del,w]) = \ord_\xi(f) + \ord_\xi(w) - 1 = \ord_\xi(w) - 1,$$
    contradicting the minimality of $\ord_\xi(w)$. Therefore, $\rad(g) = \rad(f)$. In particular, this means that there exists $n \in \NN$ such that $g$ divides $\rad(f)^n$, in other words, we have $\W(\rad(f)^n) \subseteq \W(g)$. We know that the derived series of $\W(\rad(f))$ is
    $$\W(\rad(f)) \supseteq \W(\rad(f)^2) \supseteq \W(\rad(f)^4) \supseteq \W(\rad(f)^8) \supseteq \ldots.$$
    Since $L \subseteq \W(\rad(f))$, we have $D^k(L) \subseteq D^k(\rad(f)) = \W(\rad(f)^{2^k})$. Let $m \in \NN$ such that $2^m \geq n$, so that $\W(\rad(f)^{2^m}) \subseteq \W(\rad(f)^n) \subseteq \W(g)$. It follows that $D^m(L) \subseteq \W(g) \subseteq I$. Hence, $L/I$ is solvable.
\end{proof}

\begin{cor}\label{cor:irreps are one-dim}
    Let $L$ be a subalgebra of $\W$ of finite codimension. Then all finite-dimensional irreducible representations of $L$ are one-dimensional.
\end{cor}
\begin{proof}
    Let $M$ be a finite-dimensional irreducible representation of $L$. Then $\Ann(M) = \{w \in L \mid w \cdot M = 0\}$ is an ideal of $L$ of codimension at most $\dim(M)^2 < \infty$, since $\Ann(M)$ is the kernel of the map
    \begin{align*}
        \rho \colon L &\to \mf{gl}(M) \\
        w &\mapsto \rho_w,
    \end{align*}
    where $\rho_w(m) = w \cdot m$ for $m \in M$. Therefore, $L/\Ann(M)$ is solvable by Lemma \ref{lem:L/I is solvable}. Note that $M$ is a finite-dimensional irreducible representation of $L/\Ann(M)$, so $M$ is one-dimensional by Lie's theorem.
\end{proof}

\begin{rem}
    An infinite-dimensional Lie algebra whose finite-dimensional irreducible representations are all one-dimensional is not necessarily strongly residually solvable. For example, all finite-dimensional irreducible representations of an infinite-dimensional simple Lie algebra $L$ are one-dimensional, but $L$ is not strongly residually solvable.
\end{rem}

We can now show that any subalgebra of $\W$ of finite codimension has one-dimensional extensions.

\begin{prop}\label{prop:one-dim extensions exist}
    Let $L$ be a proper subalgebra of $\W$ of finite codimension. Then $L$ has a one-dimensional extension $\overline{L}$ such that the short exact sequence of $U(L)$-modules
    $$0 \to L \to \overline{L} \to \overline{L}/L \to 0$$
    is non-split. Furthermore, any such extension can be uniquely embedded in $\W$ such that the diagram
    \begin{center}
        \begin{tikzcd}
            L \arrow[d, hook] \arrow[r, hook] & \W \\
            \overline{L} \arrow[ru, hook]     &   
        \end{tikzcd}
    \end{center}
    commutes.
\end{prop}
\begin{proof}
    We start by considering the case where $\dim(\W/L) = 1$, so that $\W$ is a one-dimensional extension of $L$. In this case, the short exact sequence
    $$0 \to L \to \W \to \W/L \to 0$$
    is non-split, since $\W$ does not have any one-dimensional $U(L)$-submodules.
    
    Now suppose $\dim(\W/L) \geq 2$. By Corollary \ref{cor:irreps are one-dim}, $\W/L$ has a one-dimensional $U(L)$-submodule $M$. Let $\pi \colon \W \to \W/L$ be the quotient map and let $\overline{L} = \pi^{-1}(M)$. Then $L \subseteq \overline{L} \subseteq \W$ and the short exact sequence
    $$0 \to L \to \overline{L} \to M \to 0$$
    is non-split for the same reason as above: $\W$ does not have any one-dimensional $U(L)$-submodules.

    The final statement is a restatement of Proposition \ref{prop:extensions are in W}.
\end{proof}

Proposition \ref{prop:one-dim extensions exist} allows us to prove Theorem \ref{thm:extending isomorphisms}.

\begin{proof}[Proof of Theorem \ref{thm:extending isomorphisms}]
    By Proposition \ref{prop:one-dim extensions exist}, there is a one-dimensional extension $L_1 \subseteq \overline{L}_1 \subseteq \W$ of $L_1$ such that the short exact sequence of $U(L)$-modules
    $$0 \to L_1 \to \overline{L}_1 \to \overline{L}_1/L_1 \to 0$$
    is non-split. Now, $L_2 \xrightarrow{\varphi^{-1}} L_1 \hookrightarrow \overline{L}_1$ is also a one-dimensional extension of $L_2$. Therefore, Proposition \ref{prop:extensions are in W} implies that there is an injective map $\overline{\varphi} \colon \overline{L}_1 \hookrightarrow \W$ such that the diagram
    \begin{center}
        \begin{tikzcd}
            L_2 \arrow[d, hook] \arrow[r, hook]       & \W \\
            \overline{L}_1 \arrow[ru, "\overline{\varphi}"', hook] &   
        \end{tikzcd}
    \end{center}
    commutes. Let $\overline{L}_2 = \overline{\varphi}(\overline{L}_1)$. Now, we have an isomorphism $\overline{\varphi} \colon \overline{L}_1 \to \overline{L}_2$ such that $\restr{\overline{\varphi}}{L_1} = \varphi$. Inducting on $\codim_{\W}(L_1)$, we conclude that $\varphi$ extends to an automorphism of $\W$.
\end{proof}

Theorem \ref{thm:extending isomorphisms} allows us to determine when two submodule-subalgebras of $\W$ are isomorphic. We now describe the automorphism group of $\W$. The most ``obvious" automorphisms of $\W$ are those which are induced by automorphisms of $\kk[t]$. We introduce notation for these automorphisms of $\W$.

\begin{ntt}
    For $n \geq -1$ and $x \in \kk$, we let $e_n(x) = (t - x)^{n+1}\del$. Letting $\alpha \in \kk^*$ and $x \in \kk$, we define a linear map
    \begin{align*}
        \rho_{x;\alpha} \colon \W &\to \W \\
        e_n &\mapsto \alpha^n e_n(x)
    \end{align*}
\end{ntt}

In fact, all automorphisms of $\W$ are of the form $\rho_{x;\alpha}$ for some $x \in \kk, \alpha \in \kk^*$. This result appeared without proof in \cite{Rudakov}. For a proof, see \cite{Bavula}.

\begin{prop}[{\cite{Rudakov}}]\label{prop:Rudakov}
    We have $\Aut(\W) = \{\rho_{x;\alpha} \mid x \in \kk, \alpha \in \kk^*\} \cong \kk \rtimes \kk^*$.
\end{prop}

As an immediate consequence, we determine exactly when two submodule-subalgebras of $\W$ are isomorphic, and compute the automorphism group of $\W(f)$.

\begin{cor}\label{cor:classification}
    Let $f,g \in \kk[t]$. Then $\W(f) \cong \W(g)$ if and only if there exist $\alpha, \gamma \in \kk^*, x \in \kk$ such that $f(s) = \gamma g(t)$, where $s = \alpha(t - x)$.
    
    Furthermore, the group of automorphisms of $\W(f)$ is
    $$\Aut(\W(f)) = \{\varphi \in \Aut(\W) \mid \varphi(f\del) = \alpha f\del \text{ for some } \alpha \in \kk^*\}.$$
\end{cor}
\begin{proof}
    Suppose there is an isomorphism $\varphi \colon \W(f) \to \W(g)$. Then Theorem \ref{thm:extending isomorphisms} and Proposition \ref{prop:Rudakov} imply that there exist $\alpha \in \kk^*, x \in \kk$ such that $\varphi = \restr{\rho_{x;\alpha}}{\W(f)}$. It follows that $\rho_{x;\alpha}(f\del)$ is a scalar multiple of $g\del$. In other words, $f(\alpha(t - x))$ is a scalar multiple of $g$.

    Conversely, suppose that there exist $\alpha \in \kk^*, x \in \kk$ such that $f(\alpha(t - x))$ is a scalar multiple of $g$. Then it is clear that $\rho_{x;\alpha}(\W(f)) = \W(g)$, so $\W(f) \cong \W(g)$.

    The computation of $\Aut(\W(f))$ is the special case $g = f$ of the above.
\end{proof}

Theorem \ref{thm:extending isomorphisms} and Proposition \ref{prop:Rudakov} imply that if $L_1$ and $L_2$ are isomorphic subalgebras of $\W$ of finite codimension, then $L_1$ is a submodule-subalgebra of $\W$ if and only if $L_2$ is also a submodule-subalgebra. This suggests that we should be able to differentiate between the subalgebras of $\W$ which are submodules and those which are not. One possible way of distinguishing between these is given in the next result\footnote{Proposition \ref{prop:codimension abelianisation} is a result of a collaboration with Jason Bell.}.

\begin{prop}\label{prop:codimension abelianisation}
    Let $L$ be a subalgebra of $\W$ of finite codimension. Then $\dim(L^{\ab}) = \codim_{\W}(L)$ if and only if $L$ is a $\kk[t]$-submodule of $\W$.
\end{prop}

The proof of Proposition \ref{prop:codimension abelianisation} will be split into two cases depending on what the associated graded algebra of $L$ looks like. We will see that if $\gr(L) \neq W_{\geq r}$ for any $r$, then it is not hard to show that $\dim(L^{\ab}) < \codim_{\W}(L)$ by simply looking at degrees of elements of $[L,L]$. Therefore, the difficult case is when $\gr(L) = W_{\geq r}$ for some $r \in \NN$.

First, we introduce some notation.

\begin{ntt}
    Let $L$ be a subalgebra of $\W$. We write $\deg(L) = \{\deg(u) \mid u \in L\}$.
\end{ntt}

We now prove the first case of Proposition \ref{prop:codimension abelianisation}, when $\gr(L) \neq W_{\geq r}$ for any $r$, in other words, when $\deg(L) \neq \ZZ_{\geq r}$ for any $r$.

\begin{lem}\label{lem:not an interval}
    Let $L$ be a subalgebra of $\W$ of finite codimension and suppose $\deg(L) \neq \ZZ_{\geq r}$ for any $r \geq -1$. Then $\dim(L^{\ab}) < \codim_{\W}(L)$.
\end{lem}
\begin{proof}
    We know that $\deg(L) \supseteq \ZZ_{\geq r}$ for some $r \in \NN$. Choose $r$ minimal with this property. As a vector space, we can write $L = L_{\geq r} \oplus \spn\{g_1\del,\ldots,g_k\del\}$, where
    $$L_{\geq r} = \spn\{f_r\del,f_{r+1}\del,f_{r+3}\del,\ldots\}$$
    is some subspace of $L$ spanned by elements of degrees $r$ and above, $\deg(f_i\del) = i$ for all $i$, and $g_i \in \CC[t]$ are polynomials such that $\deg(g_1\del) < \deg(g_2\del) < \ldots < \deg(g_k\del) < r$. By minimality of $r$, we must have $\deg(g_k\del) < r - 1$, and thus $\deg(g_1\del) < r - k$.

    Now, $\codim_{\W}(L) = r - k + 1$, so we want to show that $\dim(L/[L,L]) < r - k + 1$. Let $n_i = \deg(g_i\del)$. Taking brackets of $g_1\del$ with elements in $L_{\geq r}$, we can get elements of any degree greater than or equal to $n_1 + r$, in other words, $\ZZ_{\geq n_1 + r} \subseteq \deg([L,L])$. In $[L,L]$, we also have the elements $[g_1\del,g_i\del]$ for $i = 2,\ldots,k$. These elements have degrees $n_1 + n_i$, which are distinct integers less than $n_1 + r$. In other words,
    $$\deg([L,L]) \supseteq \ZZ_{\geq n_1 + r} \cup \{n_1 + n_2, n_1 + n_3, \ldots, n_1 + n_k\}.$$
    We therefore see that $\codim_{\W}([L,L]) \leq n_1 + r + 1 - (k - 1) = n_1 + r - k + 2$, and thus
    \begin{multline*}
        \dim(L/[L,L]) = \codim_{\W}([L,L]) - \codim_{\W}(L) \\
        \leq n_1 + r - k + 2 - (r - k + 1) = n_1 + 1 < r - k + 1,
    \end{multline*}
    as required.
\end{proof}

The remaining case is when $\gr(L) = W_{\geq r}$ for some $r \in \NN$. For this case, we will need the following easy lemma.

\begin{lem}\label{lem:derived subalgebra is principal}
    Let $L$ be a subalgebra of $\W$ of finite codimension such that $\deg(L) = \ZZ_{\geq r}$ for some $r \in \NN$ and
    $\codim_{\W}(L) = \dim(L^{\ab})$. Let $g\del$ be an element of $L$ of degree $r$. Then $[L,L] = [g\del,L]$.
\end{lem}
\begin{proof}
    Write $L = \spn\{g\del,f_1\del,f_2\del,\ldots\}$, where $\deg(f_i\del) = \deg(g\del) + i = r + i$. Then
    $$[g\del,L] = \spn\{[g\del,f_1\del],[g\del,f_2\del],[g\del,f_3\del],\ldots\},$$
    so $[g\del,L]$ has elements of degrees $2r + 1$ and above. Hence,
    $$\dim(L/[g\del,L]) = r + 1 = \codim_{\W}(L) = \dim(L^{\ab}).$$
    Since $[g\del,L] \subseteq [L,L]$, it follows that $[g\del,L] = [L,L]$.
\end{proof}

We are now ready to prove Proposition \ref{prop:codimension abelianisation}.

\begin{proof}[Proof of Proposition \ref{prop:codimension abelianisation}]
    Suppose $L$ is a submodule-subalgebra. Then $L = \W(f)$ for some $f \in \CC[t]$. We have $[L,L] = \W(f^2)$, and therefore
    $$\codim_{\W}(L) = \dim(L^{\ab}) = \deg(f).$$
    Conversely, suppose $L$ is not a submodule-subalgebra. By Lemma \ref{lem:not an interval}, we may assume that $\deg(L) = \ZZ_{\geq r}$ for some $r \in \NN$. Let $g\del \in L$ be an element of degree $r$, and write
    $$L = \spn\{g\del,g\theta_1\del,g\theta_2\del,\ldots\},$$
    where $\theta_i \in \CC(x)$. Since $g\theta_i \in \CC[x]$, the poles of $\theta_i$ are limited to roots of $g$. Furthermore, since $L$ is not a submodule-subalgebra, it must be the case that at least one $\theta_i$ has a pole at some root $\lambda$ of $g$. Let $p_i$ be the order of vanishing of $\theta_i$ at $\lambda$ (which is negative if $\theta_i$ has a pole at $\lambda$), and let $V = \{p_1,p_2,p_3,\ldots\}$. Relabeling if necessary, assume that $p_1 < p_2 < p_3 < \ldots$. In particular, $p_1 < 0$, in other words, $\theta_1$ has a pole at $\lambda$.
    
    Assume, for a contradiction, that $\dim(L/[L,L]) = \codim_{\W}(L)$. Lemma \ref{lem:derived subalgebra is principal} implies that
    $$[L,L] = [g\del,L] = \spn\{g^2\theta_i'\del \mid i \geq 1\},$$
    since $[g\del,g\theta_i\del] = g^2\theta_i'\del$. Therefore,
    $$[g\theta_i\del,g\theta_j\del] = g^2[\theta_i\del,\theta_j\del] \in \spn\{g^2\theta_i'\del \mid i \geq 1\}.$$
    It follows that $[\theta_i\del,\theta_j\del] \in \spn\{\theta_i'\del \mid i \geq 1\}$ for all $i,j$.

    Note that, for $i \neq j$, the order of vanishing of $[\theta_i\del,\theta_j\del]$ at $\lambda$ is $p_i + p_j - 1$. The above implies that $p_i + p_j - 1 = p_k - 1$ for some $k$, and therefore $p_i + p_j = p_k \in V$. Hence, $V$ is closed under addition of distinct elements. In particular, $p_1 + p_2 = p_n$ for some $n$. Since $p_1 < 0$, we must have $p_n < p_2$, which forces $n = 1$. Hence, $p_2 = 0$. Similarly, $p_1 + p_3 = p_m$ for some $m$. Now, $p_1 < 0$ implies that $m < 3$, while $p_3 > p_2 = 0$ implies that $m > 1$. Thus, $m = 2$, so $p_1 + p_3 = p_2 = 0$, meaning $p_3 = -p_1$.

    Therefore, $[\theta_1\del,\theta_3\del]$ has a pole of order 1 at $\lambda$. But $[\theta_1\del,\theta_3\del] \in \spn\{\theta_1'\del,\theta_2'\del,\theta_3'\del,\ldots\}$, leading to a contradiction. This concludes the proof.
\end{proof}

\section{Universal property}\label{sec:universal}

Up until now, all our results suggest that submodule-subalgebras of $\W$ are very rigid, in the sense that their Lie algebraic properties are controlled by $\W$. Indeed, we have shown that their derivations and automorphisms are restrictions of derivations and automorphisms of $\W$, and that their one-dimensional extensions are all contained in $\W$. Therefore, even when studying submodule-subalgebras of $\W$ as abstract Lie algebras, they still seem to remember $\W$.

This rigidity is further demonstrated in \cite{PetukhovSierra}, where the authors focused on the Poisson algebra structure of symmetric algebras of various infinite-dimensional Lie algebras. In particular, they show that the only elements of $\W^*$ which vanish on a nontrivial Poisson ideal of $S(\W)$ are what they call \emph{local functions}; these are functions given by linear combinations of derivatives at a finite set of points. They also show that the same result holds for $S(\W(f))$, where $f \in \kk[t]$: the only functions in $\W(f)^*$ which vanish on a nontrivial Poisson ideal of $S(\W(f))$ are restrictions of local functions.

In this section, we show that $\W$ can be intrinsically reconstructed from any of its subalgebras of finite codimension, purely by considering their Lie algebraic structure. In particular, we prove that $\W$ is, in the appropriate sense, a universal finite-dimensional extension of any of its subalgebras of finite codimension, which provides an explanation for the rigidity observed above.

Our goal is to state and prove the universal property satisfied by $\W$ as an extension of any of its subalgebras of finite codimension. We first define the types of extensions we want to consider.

\begin{dfn}\label{dfn:completely non-split extension}
    Let $L$ be a Lie algebra. We say that $\overline{L}$ is a \emph{completely non-split extension} of $L$ if $\overline{L}$ is a Lie algebra containing $L$ such that there exists a chain of Lie algebras
    $$L_0 = L \subseteq L_1 \subseteq L_2 \subseteq \ldots \subseteq L_n = \overline{L},$$
    where for all $i$, the representation $L_{i+1}/L_i$ of $L_i$ is finite-dimensional and irreducible, and the following short exact sequence of $L_i$-representations is non-split:
    $$0 \to L_i \to L_{i+1} \to L_{i+1}/L_i \to 0.$$
\end{dfn}

Applying Proposition \ref{prop:one-dim extensions exist} inductively, we can easily see that $\W$ is a completely non-split extension of any of its subalgebras of finite codimension.

\begin{cor}\label{cor:completely non-split extension}
    Let $L$ be a subalgebra of $\W$ of finite codimension. Then $\W$ is a completely non-split extension of $L$. \qed
\end{cor}

Note that the trivial one-dimensional representation is the only irreducible finite-dimensional representation of $\W$, since $\W$ is simple. By Proposition \ref{prop:derivations of W1},
$$\Ext_{U(\W)}^1(\kk,\W) = 0,$$
so $\W$ does not have any completely non-split extensions. Given a subalgebra $L$ of $\W$ of finite codimension, it is therefore reasonable to expect $\W$ to be the unique maximal completely non-split extension of $L$. We will see that this is indeed the case.

Thanks to Corollary \ref{cor:irreps are one-dim}, in order to study completely non-split extensions of subalgebras of $\W$, we only need to consider one-dimensional extensions. In particular, we will be able to use our results from Section \ref{sec:extensions}.

\begin{thm}\label{thm:universal property}
    Let $L$ be a subalgebra of $\W$ of finite codimension. Then $\W$ is the universal completely non-split extension of $L$, in the following sense: if $\overline{L}$ is another completely non-split extension of $L$, then $\overline{L}$ can be uniquely embedded in $\W$ such that the diagram
    \begin{center}
        \begin{tikzcd}
            L \arrow[d, hook] \arrow[r, hook] & \W \\
            \overline{L} \arrow[ru, hook]     &   
        \end{tikzcd}
    \end{center}
    commutes.
\end{thm}
\begin{proof}
    By Corollary \ref{cor:completely non-split extension}, we know that $\W$ is a completely non-split extension of $L$.

    Suppose $\overline{L}$ is another completely non-split extension of $L$. Therefore, we have a chain of Lie algebras
    $$L_0 = L \subseteq L_1 \subseteq L_2 \subseteq \ldots \subseteq L_n = \overline{L},$$
    where for all $i$, the representation $L_{i+1}/L_i$ of $L_i$ is finite-dimensional and irreducible, and the short exact sequence of $L_i$-representations
    $$0 \to L_i \to L_{i+1} \to L_{i+1}/L_i \to 0$$
    is non-split. By Corollary \ref{cor:irreps are one-dim}, $L_1/L$ is one-dimensional, so $L_1$ is a one-dimensional extension of $L$, and thus $L_1$ can be uniquely embedded in $\W$ in a way that makes the following diagram commute:
    \begin{center}
        \begin{tikzcd}
            L \arrow[d, hook] \arrow[r, hook] & \W \\
            L_1 \arrow[ru, hook]     &   
        \end{tikzcd},
    \end{center}
    by Proposition \ref{prop:extensions are in W}. Continuing inductively, we conclude that there exists an injective homomorphism $\varphi \colon \overline{L} \to \W$ such that
    \beq\label{eq:commuting triangle in theorem}
        \begin{tikzcd}
            L \arrow[d, hook] \arrow[r, hook] & \W \\
            \overline{L} \arrow[ru, hook, "\varphi"']     &   
        \end{tikzcd}
    \eeq
    commutes.

    To prove uniqueness of $\varphi$, suppose $\psi \colon \overline{L} \to \W$ is another injective homomorphism making \eqref{eq:commuting triangle in theorem} commute. Let $L_1 = \varphi(\overline{L})$ and $L_2 = \psi(\overline{L})$. These are two isomorphic subalgebras of $\W$ which contain $L$. Let $\rho = \psi \circ \varphi^{-1}$ be the isomorphism between $L_1$ and $L_2$. Since $\rho$ is an isomorphism between two subalgebras of $\W$ of finite codimension, Theorem \ref{thm:extending isomorphisms} and Proposition \ref{prop:Rudakov} imply that $\rho = \restr{\rho_{x;\alpha}}{L_1}$ for some $x \in \kk, \alpha \in \kk^*$. Note that $\restr{\rho_{x;\alpha}}{L} = \restr{\rho}{L} = \id_L$, since $\restr{\varphi}{L} = \restr{\psi}{L}$. This forces $\alpha = 1$ and $x = 0$, so that $\rho_{x;\alpha} = \rho_{0;1} = \id_{\W}$. Hence, $\psi = \varphi$.
\end{proof}

\section{Subalgebras of infinite codimension}\label{sec:infinite codimension}

Having studied derivations and extensions of subalgebras of $\W$ of finite codimension, we now consider the situation in infinite codimension. Although there are some similarities with the situation in finite codimension, we will see that there are many differences.

We will need the following subalgebras, first defined in \cite{Buzaglo}.

\begin{ntt}\label{ntt:L(f,g)}
    For $f,g \in \kk[t] \nonzero$ such that $f'g \in \kk[f]$ (in other words, $f'g = h(f)$ for some $h \in \kk[t]$), we write $L(f,g) = \kk[f]g\del$.

    Letting $g_f \in \kk[t]$ be the unique monic polynomial of minimal degree such that $f'g_f \in \kk[f]$, we write $L(f)$ instead of $L(f,g_f)$. By \cite[Proposition 4.13]{Buzaglo}, $L(f,g) \subseteq L(f)$ for all $f,g \in \kk[t]$ such that $f'g \in \kk[f]$.
\end{ntt}

Let $f,g \in \kk[t] \nonzero$ such that $f'g \in \kk[f]$. Certainly, if $\deg(f) \geq 2$, then $L(f,g)$ has infinite codimension in $\W$. In this case, $L(f,g)$ is clearly not a submodule-subalgebra of $\W$, but it is still a $\kk[f]$-submodule. Note that if $\deg(f) = 1$, then $L(f,g) = \W(g)$.

In \cite{Buzaglo}, it was shown that $L(f,g)$ is a Lie algebra and that $L(f,g)$ has a familiar Lie algebra structure: it is isomorphic to a submodule-subalgebra of $\W$.

\begin{lem}[{\cite[Lemma 4.12]{Buzaglo}}]\label{lem:L(f,g)}
    Let $f,g \in \kk[t] \nonzero$ such that $f'g = h(f) \in \kk[f]$ for some $h \in \kk[t]$. Then $L(f,g)$ is a Lie subalgebra of $\W$ and $L(f,g) \cong \W(h)$.
\end{lem}

In some sense, these are infinite codimension analogues of submodule-subalgebras: all infinite-dimensional subalgebras of $\W$ are very similar to subalgebras of the form $L(f,g)$, just like subalgebras of finite codimension are very similar to submodule-subalgebras (cf. Proposition \ref{prop:Alexey}).

\begin{thm}[{\cite{BellBuzaglo}}]\label{thm:Jason}
    Let $L$ be an infinite-dimensional subalgebra of $\W$. Then there exist $f,g \in \kk[t]$ such that $f'g \in \kk[f]$ and
    $$L(f,g) \subseteq L \subseteq L(f).$$
    In particular, $L$ has finite codimension in $L(f)$, so $L$ is isomorphic to a subalgebra of $\W$ of finite codimension.
\end{thm}

Theorem \ref{thm:Jason} implies that our results on derivations and extensions can be translated to arbitrary infinite-dimensional subalgebras of $\W$. Therefore, there are many similarities between infinite codimension and finite codimension subalgebras. However, there are also some key differences, which we highlight at the end of the section.

For example, Lemma \ref{lem:L/I is solvable} and Theorem \ref{thm:Jason} immediately imply:

\begin{cor}
    If $L$ is a proper infinite-dimensional subalgebra of $\W$, then $L$ is strongly residually solvable. \qed
\end{cor}

As a consequence of our study of derivations of submodule-subalgebras of $\W$, we can deduce that $L(f)$ has no outer derivations. In order to see this, we require the following result, which shows that for any $f \in \kk[t] \setminus \kk$, the subalgebra $L(f)$ is isomorphic to a submodule-subalgebra $\W(h)$ with $h \in \kk[t]$ a reduced polynomial.

\begin{prop}\label{prop:h_f is reduced}
    Let $f \in \kk[t] \setminus \kk$ and let $h \in \kk[t]$ such that $f'g_f = h(f) \in \kk[f]$. Then $h$ is reduced (i.e. $h$ is a scalar multiple of $\rad(h)$).
\end{prop}
\begin{proof}
    Certainly, in order to have $f'g_f = h(f)$, it must be the case that $f'$ divides $h(f)$. In particular, this implies that $h$ must vanish on $f(V(f')) = \{f(\lambda) \mid f'(\lambda) = 0\}$. We claim that these are all the roots of $h$, in other words, $h$ is a scalar multiple of the polynomial
    $$\widetilde{h} = \prod_{\lambda \in f(V(f'))}(t - \lambda).$$
    This will certainly imply that $h$ is reduced, as required.
    
    In the proof of \cite[Proposition 4.13]{Buzaglo}, it was shown that $h$ is the generator of the ideal
    $$I(f) \coloneqq \{p \in \kk[t] \mid f' \text{ divides } p(f)\}$$
    of $\kk[t]$. As mentioned in the first paragraph, any element of $I(f)$ must vanish on $f(V(f'))$, so $\widetilde{h}$ divides $h$. Therefore, it suffices to show that $\widetilde{h} \in I(f) = (h)$, that is, $f'$ divides $\widetilde{h}(f) = \prod_{\lambda \in f(V(f'))}(f - \lambda)$.

    Let $\lambda \in V(f')$. Then $(t - f(\lambda))$ divides $\widetilde{h}$ by definition, so $f_\lambda \coloneqq f - f(\lambda)$ divides $\widetilde{h}(f)$. Note that $\ord_\lambda(f_\lambda') = \ord_\lambda(f_\lambda) - 1$, since $f_\lambda$ vanishes at $\lambda$. But $f_\lambda' = f'$, so this means that $\ord_\lambda(f') = \ord_\lambda(f_\lambda) - 1$. Therefore,
    $$\ord_\lambda(\widetilde{h}(f)) \geq \ord_\lambda(f_\lambda) > \ord_\lambda(f')$$
    for all $\lambda \in V(f')$. This implies that $f'$ divides $\widetilde{h}(f)$, as claimed.
\end{proof}

Theorem \ref{thm:derivations of W(f)}, Lemma \ref{lem:L(f,g)}, and Proposition \ref{prop:h_f is reduced} immediately imply that $L(f)$ has no outer derivations.

\begin{cor}
    Let $f \in \kk[t] \setminus \kk$. Then all derivations of $L(f)$ are inner. In other words, $H^1(L(f);L(f)) = 0$. \qed
\end{cor}

Furthermore, we can compute derivations of arbitrary infinite-dimensional subalgebras of $\W$, generalising Corollary \ref{cor:derivations}. This further highlights the rigidity exhibited by $\W$: derivations of any infinite-dimensional subalgebra of $\W$ are controlled by $\W$.

\begin{thm}\label{thm:derivations of general subalgebras}
    Let $L$ be an infinite-dimensional subalgebra of $\W$. Then
    $$\Der(L) = \{\ad_w \mid w \in \W \text{ such that } \ad_w(L) \subseteq L\} \cong N_{\W}(L).$$
    In other words, derivations of $L$ are restrictions of derivations of $\W$. Consequently, $H^1(L;L) \cong N_{\W}(L)/L$.
\end{thm}
\begin{proof}
    Let $f \in \kk[t]$ such that $L$ has finite codimension in $L(f)$, which exists by Theorem \ref{thm:Jason}. Let $h \in \kk[t]$ such that $f'g_f = h(f)$, so that $L(f) \cong \W(h)$. Let $\varphi \colon L(f) \to \W(h)$ be an isomorphism, and let $\widetilde{L} = \varphi(L) \subseteq \W(h)$, which is a subalgebra of $\W$ of finite codimension.

    Let $d \in \Der(L)$, and let $\widetilde{d} = \varphi \circ d \circ \varphi^{-1} \in \Der(\widetilde{L})$. By Corollary \ref{cor:derivations}, $\widetilde{d} = \ad_{\widetilde{w}}$ for some $\widetilde{w} \in \W$ such that $\ad_{\widetilde{w}}(\widetilde{L}) \subseteq \widetilde{L}$. We claim that $\widetilde{w} \in \W(h)$.

    By Proposition \ref{prop:Alexey}, there exists $g \in \kk[t]$ such that $\W(g) \subseteq \widetilde{L} \subseteq \W(\rad(g))$. Furthermore, $\W(g) \subseteq \widetilde{L} \subseteq \W(h)$, so $h$ divides $g$. By Proposition \ref{prop:h_f is reduced}, $h$ is reduced, so $h$ divides $\rad(g)$. Therefore,
    $$\W(g) \subseteq \widetilde{L} \subseteq \W(\rad(g)) \subseteq \W(h).$$
    Hence, in order to prove the claim, it suffices to prove that $\widetilde{w} \in \W(\rad(g))$, in other words, that $\ord_\xi(\widetilde{w}) \geq 1$ for all $\xi \in V(g)$.

    Assume, for a contradiction, that $\ord_\xi(\widetilde{w}) = 0$ for some $\xi \in V(g)$. Let $\widetilde{x} \in \widetilde{L}$ such that $\ord_\xi(\widetilde{x})$ is minimal. Note that $\ord_\xi(\widetilde{x}) \geq 1$, since $\widetilde{x} \in \W(\rad(g))$. By Lemma \ref{lem:order of vanishing}, $\ord_\xi([\widetilde{w},\widetilde{x}]) = \ord_\xi(\widetilde{x}) - 1$. But $[\widetilde{w},\widetilde{x}] \in \widetilde{L}$, since $\widetilde{w} \in N_{\W}(\widetilde{L})$, contradicting the minimality of $\ord_\xi(\widetilde{x})$. This proves the claim.

    Now, we have
    $$d = \varphi^{-1} \circ \widetilde{d} \circ \varphi = \varphi^{-1} \circ \ad_{\widetilde{w}} \circ \varphi.$$
    Let $w = \varphi^{-1}(\widetilde{w}) \in L(f)$. Then,
    $$d(x) = \varphi^{-1}([\widetilde{w},\varphi(x)]) = [w,x] = \ad_w(x)$$
    for all $x \in L$. Hence, $d = \ad_w$, which concludes the proof.
\end{proof}

\begin{rem}
    In the proof of Theorem \ref{thm:derivations of general subalgebras}, it was shown that if $L$ is a subalgebra of $\W$, and $f \in \kk[t]$ is a reduced polynomial such that $L \subseteq \W(f)$, then $N_{\W}(L) \subseteq \W(f)$.
\end{rem}

Another consequence of our previous work is that any infinite-dimensional subalgebra $L \subseteq \W$ has a universal completely non-split extension isomorphic to $\W$, which follows by combining Theorems \ref{thm:universal property} and \ref{thm:Jason}.

\begin{cor}
    Let $L$ be an infinite-dimensional subalgebra of $\W$. Then $L$ has a universal completely non-split extension $\overline{L}$: if $X$ is another completely non-split extension of $L$ then $X$ embeds in $\overline{L}$ such that the following diagram commutes:
    \begin{center}
        \begin{tikzcd}
            {L} \arrow[r, hook] \arrow[d, hook]   & \overline{L} \\
            X\arrow[ru, hook] &   
        \end{tikzcd}.
    \end{center}
    Furthermore, $\overline{L} \cong \W$ as Lie algebras. \qed
\end{cor}

We now highlight the difference between the situations in finite and infinite codimension. Consider an infinite-dimensional subalgebra $L \subseteq \W$ of infinite codimension. Theorem \ref{thm:Jason} shows that $L$ is isomorphic to a subalgebra of $\W$ of finite codimension. Considering the universal property of Theorem \ref{thm:universal property}, it is natural to ask if the universal completely non-split extension $\overline{L}$ of $L$ embeds in $\W$ in a way that makes the diagram
\beq\label{eq:diagram does not embed}
    \begin{tikzcd}
        L \arrow[r, hook] \arrow[d, hook]   & \W \\
        \overline{L} \arrow[ru, hook] &   
    \end{tikzcd}
\eeq
commute. We show that this is not the case.

\begin{prop}\label{prop:no embedding}
    Let $L$ be an infinite-dimensional subalgebra of $\W$ of infinite codimension. Then the universal completely non-split extension $\overline{L}$ of $L$ cannot be embedded in $\W$ in a way that makes \eqref{eq:diagram does not embed} commute.
\end{prop}
\begin{proof}
    Suppose we can embed $\overline{L}$ in $\W$ such that \eqref{eq:diagram does not embed} commutes. Certainly, the image of $\overline{L}$ in $\W$ must have infinite codimension in $\W$. Therefore, the map
    $$\W \cong \overline{L} \hookrightarrow \W$$
    is a nonzero endomorphism of $\W$ which is not an automorphism. By \cite[Theorem 2.3]{Du}, this is impossible.
\end{proof}

\section{Finite codimension subalgebras of the full Witt algebra}\label{sec:Witt}

We define submodule-subalgebras of $W$ analogously to those of $\W$: they are Lie subalgebras of $W$ which are also submodules over $\kk[t,t^{-1}]$. They are denoted by $W(f) \coloneqq fW$, where $f \in \kk[t,t^{-1}]$. Since submodule-subalgebras of $W$ correspond to ideals of $\kk[t,t^{-1}]$, we will usually assume that $f \in \kk[t]$ and that $f(0) \neq 0$.

In this section, we prove results about derivations and extensions of submodule-subalgebras of the Witt algebra similar to those of $\W$. Most proofs are omitted, since they are identical to those in Sections \ref{sec:general submodule-subalgebras}--\ref{sec:universal}.

\subsection{Derivations}

We start by considering derivations of submodule-subalgebras of the Witt algebra, similarly to Section \ref{sec:general submodule-subalgebras}.

\begin{thm}\label{thm:Witt submodule-subalgebra derivations}
    Let $f \in \kk[t]$ such that $f(0) \neq 0$. We have $\Der(W(f)) = \{\ad_w \mid w \in W(\rad(f))\}$. Consequently, $\Der(W(f)) \cong W(\rad(f))$, and
    $$\dim H^1(W(f);W(f)) = \deg(f) - \deg(\rad(f)).$$
\end{thm}

We firstly note that the notions of degree and $d$-compatibility are still well-defined for derivations of submodule-subalgebras of $W$. In other words, for $d \in \Der(W(f),W)$, we define
$$\deg(d) = \max\{\deg(d(x)) - \deg(x) \mid x \in W(f)\}.$$
Since $W(f)$ is a finitely generated Lie algebra, $\deg(d)$ is well-defined by a similar proof to Lemma \ref{lem:degree of derivation is bounded}. Similarly, for $d$-compatibility we simply note that an analogous result to Lemma \ref{lem:d-compatible} holds in this situation.

One key difference between submodule-subalgebras of $W$ and of $\W$ is that the associated graded algebra of $W(f)$ is $W$. Therefore, in order to compute the associated graded derivation of $d \in \Der(W(f),W)$, we only need to know about $\Der(W)$, rather than $\Der(W_{\geq n},W)$ as was the case in Section \ref{sec:general submodule-subalgebras}, resulting in a much simpler picture.

The computation of $\Der(W)$ is similar to that of $\Der(\W)$ in Proposition \ref{prop:derivations of W1}: we simply apply Theorem \ref{thm:Farnsteiner}.

\begin{prop}
    We have $\Der(W) = \Inn(W)$. \qed
\end{prop}

This has been known for a long time. For example, a proof can be found in \cite{IkedaKawamoto}, where the authors compute derivations of a class of Lie algebras known as \emph{generalised Witt algebras}.

Now we can define associated graded derivations similarly to Section \ref{sec:general submodule-subalgebras}. The proof is identical to Proposition \ref{prop:associated graded derivation}.

\begin{prop}\label{prop:Witt associated graded derivation}
    Let $d \in \Der(W(f),W)$ and define a $\kk$-linear map $\gr(d) \colon W \to W$ by
    $$\gr(d)(e_k) = \begin{cases}
        \LT(d(x)), &\text{where } x \in W(f) \text{ and }\LT(x) = e_k, \text{ if } k \text{ is } d\text{-compatible,} \\
        0, &\text{if } k \text{ is not } d\text{-compatible,}
    \end{cases}$$
    for $k \in \ZZ$. Then $\gr(d) = \lambda \ad_{e_k} \in \Der(W)$ for some $\lambda \in \kk$ and $k \in \ZZ$. \qed
\end{prop}

The next result is analogous to Proposition \ref{prop:derivations into W}.

\begin{prop}\label{prop:Witt derivation rational function}
    Let $d \in \Der(W(f),W)$. Then $d = \ad_{\frac{h}{f}\del}$ for some $h \in \kk[t,t^{-1}]$. \qed
\end{prop}

We can now prove Theorem \ref{thm:Witt submodule-subalgebra derivations}.

\begin{proof}[Proof of Theorem \ref{thm:Witt submodule-subalgebra derivations}]
    The proof follows in a similar way to Theorem \ref{thm:derivations of W(f)}. The only difference is that in $W$ we allow poles at 0, so we need to be a bit more careful. Note that the proof of Theorem \ref{thm:derivations of W(f)} only considers $\ord_\xi$ for $\xi \in V(f)$. Since we assumed that $f(0) \neq 0$, we have $0 \not\in V(f)$, so we do not run into any problems by allowing poles at 0.
\end{proof}

\subsection{Extensions}

We proceed similarly to Section \ref{sec:extensions} to compute one-dimensional extensions of submodule-subalgebras of the Witt algebra. Just like submodule-subalgebras of $\W$, one-dimensional extensions of $W(f)$ are either contained in $W(\rad(f))$ or they come from non-split extensions of $W/W(t - \xi)$ by $W(f)$ for $\xi \in V(f)$ with $\ord_\xi(f) = 1$. We explain the proof of this result, with emphasis on the difference between this situation and the proof of Theorem \ref{thm:one-dimensional extension}.

\begin{thm}\label{thm:Witt extensions}
    Let $f \in \kk[t]$ with $f(0) \neq 0$. If $M \cong W/W(t - \xi)$ for some $\xi \in V(f)$ with $\ord_\xi(f) = 1$, then
    $$\dim(\Ext_{U(W(f))}^1(W/W(t - \xi),W(f))) = 1.$$
    In this case, the unique non-split extension of $W/W(t - \xi)$ by $W(f)$ is $W(\frac{f}{t - \xi})$. If $M \cong \kk$ is trivial, we have
    $$\dim(\Ext_{U(W(f))}^1(\kk,W(f))) = \deg(f) - \deg(\rad(f)).$$
    Otherwise, $\Ext_{U(W(f))}^1(M,W(f)) = 0$.
\end{thm}

As in Section \ref{sec:extensions}, we need to compute the derived subalgebra of $W(f)$ for $f \in \kk[t]$. Just like in Lemma \ref{lem:derived subalgebra}, we can easily see that $[W(f),W(f)] = W(f^2)$.

\begin{lem}
    For $f \in \kk[t,t^{-1}]$, the derived subalgebra of $W(f)$ is $W(f^2)$. \qed
\end{lem}

We now need to show that one-dimensional extensions of submodule-subalgebras of $W$ are contained in $W$. The proof of this is nearly identical to Proposition \ref{prop:extensions are in W}.

\begin{prop}\label{prop:Witt extensions are in W}
    Let $L$ be a Lie subalgebra of $W$ of finite codimension, and let $\overline{L}$ be a one-dimensional extension of $L$. Then $\overline{L}$ can be uniquely embedded in $W$ such that the diagram
    \begin{center}
        \begin{tikzcd}
            L \arrow[d, hook] \arrow[r, hook] & W \\
            \overline{L} \arrow[ru, hook]     &   
        \end{tikzcd}
    \end{center}
    commutes. \qed
\end{prop}

The only subtlety in the proof of Proposition \ref{prop:Witt extensions are in W} is that, when applying \cite[Proposition 3.2.7]{PetukhovSierra} to deduce that there exists $f \in \kk[t]$ such that
$$W(f) \subseteq L \subseteq W(\rad(f)),$$
we need to assume that $f(0) \neq 0$. After this, the proof is exactly the same as Proposition \ref{prop:extensions are in W}, since we never consider any orders of vanishing at 0. The proof of Theorem \ref{thm:Witt extensions} now follows similarly to Theorem \ref{thm:one-dimensional extension}.

As we did with $\W$, we summarise the computation of one-dimensional extensions of submodule-subalgebras of $W$.

\begin{cor}
    Suppose $\overline{L}$ is a one-dimensional extension of $W(f)$, where $f \in \kk[t]$ and $f(0) \neq 0$. Then either $\overline{L} \subseteq W(\rad(f))$ or $\overline{L} = W(\frac{f}{t - \xi})$, where $\xi \in V(f)$ with $\ord_\xi(f) = 1$. \qed
\end{cor}

\subsection{Isomorphisms between subalgebras of the Witt algebra of finite codimension}

As we did in Section \ref{sec:isomorphism}, we explain how any isomorphism between subalgebras of $W$ of finite codimension extends to an automorphism of $W$, analogously to Theorem \ref{thm:extending isomorphisms}.

\begin{thm}\label{thm:Witt isomorphism}
    Let $L_1$ and $L_2$ be subalgebras of $W$ of finite codimension and suppose there is an isomorphism of Lie algebras $\varphi \colon L_1 \to L_2$. Then $\varphi$ extends to an automorphism of $W$.
\end{thm} 

The proof of Theorem \ref{thm:Witt isomorphism} follows by showing that any subalgebra of $W$ of finite codimension has a one-dimensional extension contained in $W$. As we did with $\W$, we can achieve this by showing that any such subalgebra of $W$ is strongly residually solvable. The proof is exactly the same as Lemma \ref{lem:L/I is solvable}.

\begin{lem}
    If $L$ is a proper subalgebra of $W$ of finite codimension then $L$ is strongly residually solvable. \qed
\end{lem}

Applying Lie's theorem, we deduce that all finite-dimensional irreducible representations of $L$ are one-dimensional, just like Corollary \ref{cor:irreps are one-dim}.

\begin{cor}
    Let $L$ be a subalgebra of $W$ of finite codimension. Then all finite-dimensional irreducible representations of $L$ are one-dimensional. \qed
\end{cor}

We can now show that any subalgebra of $W$ of finite codimension has a one-dimensional extension. The proof is identical to that of Proposition \ref{prop:one-dim extensions exist}.

\begin{prop}\label{prop:one-dimesional extensions exist}
    Let $L$ be a proper subalgebra of $W$ of finite codimension. Then $L$ has a one-dimensional extension $\overline{L}$ such that the short exact sequence of $U(L)$-modules
    $$0 \to L \to \overline{L} \to \overline{L}/L \to 0$$
    is non-split. Furthermore, by Proposition \ref{prop:Witt extensions are in W}, all such extensions can be uniquely embedded in $W$ such that the diagram \begin{center}
        \begin{tikzcd}
            L \arrow[d, hook] \arrow[r, hook] & W \\
            \overline{L} \arrow[ru, hook]     &   
        \end{tikzcd}
    \end{center}
    commutes. \qed
\end{prop}

Automorphisms of $W$ are different to those of $\W$: in $W$, we no longer have access to the map $e_n \mapsto e_n(x)$ for $x \in \kk \nonzero$, since $(t - x)^n\del \not\in W$ for $n < 0$. Furthermore, in $\kk[t,t^{-1}]$ we have an automorphism $t \mapsto t^{-1}$, which induces an automorphism of $W$.

\begin{ntt}
    For $\alpha \in \kk^*$. Define linear maps $\sigma_\alpha, \tau \colon W \to W$ by
    $$\sigma_\alpha(e_n) = \alpha^n e_n, \quad \tau(e_n) = -e_{-n},$$
    for $n \in \ZZ$.
\end{ntt}

Similarly to $\W$, these give rise to all automorphisms of $W$.

\begin{prop}[{\cite[Theorem 1.1]{BavulaWitt}}]
    We have $\Aut(W) = \{\sigma_\alpha \mid \alpha \in \kk^*\} \rtimes \{\tau^{\pm 1}\}$.
\end{prop}

This allows us to determine when two submodule-subalgebras of the Witt algebra are isomorphic. It is worth noting that, unlike submodule-subalgebras of $\W$, we have $W(f) = W(t^nf)$ for $f \in \kk[t,t^{-1}]$ and $n \in \ZZ$.

\begin{cor}\label{cor:Witt isomorphism classes}
    Let $f,g \in \kk[t]$ such that $f(0) \neq 0$ and $g(0) \neq 0$ and suppose $W(f) \cong W(g)$ as Lie algebras. Then one of the following is true:
    \begin{enumerate}
        \item There exist $\alpha, \beta \in \kk^*$ such that $f(\alpha t) = \beta g(t)$;
        \item There exist $\alpha, \beta \in \kk^*$ and $n \in \NN$ such that $t^n f(\alpha t^{-1}) = \beta g(t)$. \qed
    \end{enumerate}
\end{cor}

\subsection{Universal property}

We proceed similarly to Section \ref{sec:universal} to prove that the Witt algebra satisfies a similar universal property to $\W$. We easily see that $W$ is a completely non-split extension of any of its subalgebras of finite codimension by applying Proposition \ref{prop:one-dimesional extensions exist} inductively.

\begin{cor}
    Let $L$ be a subalgebra of $W$ of finite codimension. Then $W$ is a completely non-split extension of $L$. \qed
\end{cor}

Just like $\W$, the Witt algebra can be reconstructed from any of its subalgebras of finite codimension as the universal completely non-split extension.

\begin{thm}\label{thm:universal property W}
    Let $L$ be a subalgebra of $\W$ of finite codimension. Then $\W$ is the universal completely non-split extension of $L$, in the following sense: if $\overline{L}$ is another completely non-split extension of $L$, then $\overline{L}$ can be uniquely embedded in $\W$ such that the diagram
    \begin{center}
        \begin{tikzcd}
            L \arrow[d, hook] \arrow[r, hook] & \W \\
            \overline{L} \arrow[ru, hook]     &   
        \end{tikzcd}
    \end{center}
    commutes.
\end{thm}
\begin{proof}
    The proof follows by an inductive application of Proposition \ref{prop:Witt extensions are in W}, similarly to the proof of Theorem \ref{thm:universal property}. As before, uniqueness of the embedding $\overline{L} \hookrightarrow W$ follows from Theorem \ref{thm:Witt isomorphism}.
\end{proof}

We remark that we do not have a classification of subalgebras of $W$ of infinite codimension, so we do not have any analogous results to Section \ref{sec:infinite codimension} for the Witt algebra. This is the subject of ongoing research.

\end{document}